\documentclass[final,11pt]{paper}
\usepackage[utf8]{inputenc}
\newcommand{\dx}{ {~\rm d}x}
\newcommand{\dt}{ {~\rm d}t}
\renewcommand{\d}{{~\rm d}\!}
\usepackage{amsmath}
\usepackage{amssymb,color}
\usepackage[color,notref,notcite]{showkeys}
\usepackage{tikz}
\definecolor{labelkey}{rgb}{0.6,0,1}

\def\epsilon{\varepsilon}

\def\R{{\mathbb R}}
\def\Div{\mathop{\rm div}\nolimits}

\usepackage[normalem]{ulem}
\newcounter{corr}
\definecolor{violet}{rgb}{0.580,0.,0.827}
\newcommand{\corr}[3]{\typeout{Warning : a correction remains in page
\thepage}
				\stepcounter{corr}        
				{\color{blue}\ifmmode\text{\,\sout{\ensuremath{#1}}\,}\else\sout{#1}\fi}
       {\color{red}#2}
       {\color{violet} #3}}
\usepackage{graphics}

\newtheorem{theorem}{Theorem}[section]
\newtheorem{lemma}[theorem]{Lemma}

\newtheorem{remark}[theorem]{Remark}
\newtheorem{definition}[theorem]{Definition}

\usepackage{constants}
\newconstantfamily{ctrcst}{symbol=C,reset={chapter}}
\def\ctel#1{\ensuremath{\Cl[ctrcst]{#1}}}
\def\cter#1{\ensuremath{\Cr{#1}}}

\newcommand{\eop}{{\unskip\nobreak\hfil\penalty50
           \hskip2em\hbox{}\nobreak\hfil\mbox{\rule{1ex}{1ex} \qquad}
   \parfillskip=0pt
   \finalhyphendemerits=0\par\medskip}}
\newenvironment{proof}[1][]{\noindent {\bf Proof#1. } }{\eop}

\begin{document}

\title{Weighted $p-$Laplace approximation of linear and quasi-linear elliptic problems with measure data}

\author{Robert Eymard, David Maltese and Alain Prignet\thanks{LAMA, Univ. Gustave Eiffel, Univ. Paris Est Créteil, CNRS, F-77454 Marne-la-Vallée, France, {\tt robert.eymard, david.maltese, alain.prignet@univ-eiffel.fr}}}

\maketitle

\abstract{We approximate the solution to some linear and 
degenerate quasi-linear problem involving a linear elliptic operator
(like the semi-discrete in time implicit Euler approximation of Richards and Stefan
equations)
with measure right-hand
side and heterogeneous anisotropic diffusion matrix.
This approximation is obtained through the addition of a weighted $p-$Laplace term. A well chosen diffeomorphism between $\R$ and $(-1,1)$ is used for the estimates of the approximated solution, and is involved in the above weight. We show that this approximation converges to a weak sense of the problem for general right-hand-side, and to the entropy solution in the case where the right-hand-side is in $L^1$.  }

\section{Introduction}

This paper is focused on the approximation of a solution
of second order linear and quasilinear elliptic equations in divergence form with
coefficients in $L^\infty(\Omega)$
($\Omega \subset \R^N,\ N\in\mathbb{N},\,N\ge 2$ is an
open bounded subset)
and measure data. The obtained result provides the existence in the quasilinear case, and a uniquess result is also given for $L^1$ right-hand side. The linear problem is to find a  measurable
function $u$ defined on $\Omega$ such that, in some senses which will be given
below, the following holds:
\begin{equation}\label{eq:lin:probcont}
- \Div ( \Lambda   \nabla u) = f  ~\text{in}~\Omega,
\end{equation}
together with the homogeneous Dirichlet boundary condition
\begin{equation}\label{eq:probbound}
u= 0~\text{on}~ \partial \Omega.
\end{equation}
The quasilinear problem consists in finding a pair of measurable functions $(b,u)$ such that the following relations hold:
\begin{equation}\label{eq:probcont}
b - \text{div} ( \Lambda   \nabla u) = f  ~\text{in}~\Omega,
\end{equation}
completed by the following relation:
\begin{equation}\label{eq:probcontbetazeta}
\hbox{There exists }v\hbox{ measurable on }\Omega\hbox{ such that }b = \beta(v) \hbox{ and } u = \zeta(  v ) ~\text{ a.e. in}~\Omega,
\end{equation}
where $\beta$ and $\zeta$ are nonstrictly increasing functions (precise assumptions on these functions are given by \eqref{eq:hypbetazeta} in Section \ref{sec:nonlinear}). The quasilinear framework includes a semi-discrete in time version of some degenerate equations such as the Richards or the Stefan equations, as precised in Section  \ref{sec:nonlinear}.
The quasilinear problem is supplemented with the boundary condition \eqref{eq:probbound}.

The following assumptions are made on the data $\Lambda$, $f$.
\begin{subequations}
\begin{align}\bullet~ & \Lambda \in L^\infty(\Omega)^{ N \times N} \mbox{ is symmetric and  there exists $\underline{\lambda},\overline{\lambda}>0$ such that, for a.e. $x\in\Omega$,}\nonumber\\
&\mbox{ and, for all $\xi\in {\R}^N$, $\underline{\lambda} |\xi|^2\le  \Lambda(x)\xi\cdot\xi\le \overline{\lambda} |\xi|^2$,}\label{hyplambda}\\
\bullet~  & f \in M(\Omega). \label{eq:hypfmesure}
\end{align}
\label{eq:hypomlambf}
\end{subequations}
In \eqref{eq:hypfmesure}, $M(\Omega)$ denotes the Radon measures set,
defined as the dual space of the continuous functions on
$\overline{\Omega}$ with its classical norm.
Note that, in the case
$N=1$, there holds $M(\Omega)\subset H^{-1}(\Omega)$; then these problems enter into
the framework of \cite{dro2020high}; we therefore consider here only $N\ge 2$. 
We could consider as well the case where a term $\text{div} F$ with $F\in L^2(\Omega)^N$ is added to $f$, since the same results as those obtained in this paper also hold in this case.

\medskip

Let us recall a few results concerning the linear problem \eqref{eq:lin:probcont}.
\begin{itemize}
 \item The existence of a weak solution in the sense of Definition \ref{def:weaksol} for any $N\ge 2$ is given in \cite{stam1965pro} (details of this result are given in \cite{pri1995rem}).
 \item Its uniqueness is proved for $N=2$ in \cite{gal1994exi} for general diffusion fields: the proof relies on a regularity result \cite{meyers1963est} which holds on domains $\Omega$ with $C^2$ boundary, extended in \cite{gal1999reg} to all domains with Lipschitz boundaries.
 \item In the case $N\ge 3$, this uniqueness result remains true if $\Lambda$ is regular enough to apply the arguments of  Agmon, Douglis and Nirenberg in the duality proof provided by \cite{stam1965pro}, but it is no longer true for general diffusion fields: indeed, in \cite{pri1995rem}, it is shown that, for a particular diffusion field $\Lambda$ inspired by \cite{ser1964path} (see \cite{pri1995rem} for more details), there exist infinitely many non-zero weak solutions $u$ to Problem  \eqref{eq:lin:probcont}-\eqref{eq:probbound} (in the sense of Definition \ref{def:weaksol}) for $N=3$, even with $f\equiv 0$. 
\end{itemize}

As in \cite{bg1989nonlin,boc1992nonlinear}, we consider solutions which are limit of sequences of regularised problems. Such solutions can be characterised by adding conditions in the definition of a weak solution, when the right-hand side is in $L^1(\Omega)$, and a uniqueness result can be proved. This is done by the notion of entropy weak sense \cite{ben1995theory}, explored by several approaches in the literature (among them renormalised solutions by Lions and Murat, see \cite{dal1999ren}).


This sense is provided by Definition \ref{def:lin:entsol} in the linear case, and a straightforward adaptation in the quasilinear case is given by Definition \ref{def:qlentsol}.


\medskip

The proofs of the existence of a solution in a weak sense and in an entropy weak sense are done in this paper in a different way from \cite{bg1989nonlin,boc1992nonlinear,ben1995theory}, where the existence is obtained through the regularization of the right-hand side. Here we keep the right-hand-side measure unchanged, hence remaining in $(W^{1,p}_0(\Omega))'$ for any $p>N$. A natural idea would be to add a vanishing $p-$Laplace regularisation term (we show below in Section \ref{sec:motiv} that we need in fact a weighted one). The existence of a solution to the regularised problem will then be obtained through the use of a fix-point method. 
As in \cite{BGV}, a diffeomorphism between $\mathbb{R}$ and an open bounded interval is used for deriving estimates.
We follow a similar technique to \cite{eym2021ell}, consisting in using the diffeomorphism $\psi~:~\R\to (-1,1)$, defined by
\begin{align} \label{eq:def:psi }
  &\begin{array}{l| l} \displaystyle
    \psi : \quad & \quad \R \longrightarrow (-1,1)
		\\ [1ex] & \displaystyle \quad s \longmapsto\frac {\ln(1+|s|)} {1+\ln(1+|s|)}{\rm sign}(s),
  \end{array}
  \end{align}
  where ${\rm sign}(s)=1$ if $s\ge 0$ and $-1$ if $s<0$. Note that $\psi$ is an odd strictly increasing function such that $\psi'(s)\in(0,1]$ for all $s\in \R$.

 The advantage of this diffeomorphism over the one used in \cite{BGV} is that it does not introduce a supplementary parameter which must vary in order that the weak sense be fully satisfied.
 
 We show in Section \ref{sec:motiv} why it seems necessary to introduce a weighted dependence with respect to this function in the $p-$Laplace stabilisation term. 
 We then study this regularised problem in Section \ref{sec:regpb}, where we show the existence of a solution using Schaefer's theorem (which is a variant of the Leray-Schauder fix-point method). We then show some estimates on a solution to this regularised problem, enabling a convergence proof to the weak and entropy weak solutions of the linear problem (see Section \ref{sec:linear}). Similar proofs are then derived in the case of the quasilinear problem  in Section \ref{sec:nonlinear}, in which the nonlinear dependence between $b$ and $u$ is handled through Minty's trick.

 \pagebreak
 
{\bf Notes applying to the whole paper:}
\begin{itemize}
\item We fix a given value $p\in (N,+\infty)$ for the whole paper, which implies that $W^{1,p}_0(\Omega) \subset C(\overline{\Omega})\subset L^\infty(\Omega)$ and by duality, $M(\Omega)\subset (W^{1,p}_0(\Omega))'$.
 \item We denote for short $\Vert\cdot\Vert_p$ instead of
$\Vert\cdot\Vert_{L^p(\Omega)}$ or $\Vert\cdot\Vert_{L^p(\Omega)^N}$.
\item We use, for a.e. $x\in\Omega$, $\Lambda(x)$ as a linear operator from $\mathbb{R}^N$ to $\mathbb{R}^N$, which applies to 
the element of $\mathbb{R}^N$ which immediately follows.
\item We use a few times Sobolev inequalities \cite{adams1975sob}: we denote $C_{{\rm sob}}^{(r,q)}$, also depending on $N$ and $|\Omega|$, such that
\begin{equation}\label{eq:sobolev}
\| u \|_q \le C_{{\rm sob}}^{(r,q)} \| \nabla u \|_r,\mbox{ for any }u \in W^{1,r}_0(\Omega),
\end{equation}
for any $r\in [1,+\infty)$ and $ q \in  [1,\frac{r N}{N-r} ]$ if $r < N$,  $ q \in [1,+\infty)$ if $r=N$ and $ q \in [1,+\infty]$ if $r>N$.
\end{itemize}

\section{Motivation for the definition of the regularised problem}\label{sec:motiv}

This section aims to motivate the choice of the nonlinear weight $\alpha$ introduced in the vanishing $p-$Laplacian term:
\[
 - \Div ( \Lambda \nabla u_{\epsilon} +  \epsilon\  \alpha(u_{\epsilon}) |\nabla u_{\epsilon}|^{p-2}\nabla u_{\epsilon}) = f.
\]
In the preceding equation, $\epsilon>0$ is meant to tend to zero, and $\alpha$ is a positive function to be chosen such that we can prove the following properties:
\begin{itemize}
\item there exists at least one solution $u_\epsilon$ to the regularised problem;
 \item using the function $\psi (u_\epsilon)$ as a test function where $\psi$ is defined by \eqref{eq:def:psi }, we can derive estimates on  $u_\epsilon$ independent of $\epsilon$, enabling to prove that any limit of $u_\epsilon$ as $\epsilon\to 0$ belong to the functional spaces containing the solutions to the weak or entropy weak sense of the problem;
 \item the vanishing term indeed vanishes as $\epsilon$ tends to $0$.
\end{itemize}
 
Let us consider $\psi (u_\epsilon)$ as a test function in the regularised problem.
Using the positivity of the term $ \epsilon\  \alpha(u_{\epsilon}) |\nabla u_{\epsilon}|^{p-2}\nabla u_{\epsilon}\cdot\nabla\psi (u_\epsilon)$, we get in a similar way to \cite{BGV} the following result: 
\[
 \Vert \nabla u_{\epsilon}\Vert_q \le C,\hbox{ forall }q\in(1,\frac N {N-1}),
\]
and therefore, from Sobolev inequalities, that
\[
 \Vert u_{\epsilon}\Vert_{\hat q} \le C,\hbox{ forall }{\hat q}\in I_N,
\]
where $I_N = (1,\frac N {N-2})$ if $N>2$ and $I_N = (1,+\infty)$ if $N=2$.
The function $\psi$ defined by \eqref{eq:def:psi } enables in particular the following estimate:
\begin{equation}\label{eq:estpsiprimeun}
 \left\Vert \frac {1} {\psi' (u_{\epsilon})}\right\Vert_{\hat q} \le C.
\end{equation}
We also obtain the following inequality
$$
\epsilon \int_\Omega \alpha(u_\epsilon)\psi' (u_\epsilon)|\nabla u_\epsilon|^{p} \dx \le C.
$$
Using the preceding inequalities must be sufficient to prove that, for any function $w\in C^\infty_c(\Omega)$,
$$
\lim_{\epsilon\to 0} \epsilon \int_\Omega \alpha(u_\epsilon)|\nabla u_\epsilon|^{p-2} \nabla u_\epsilon \cdot \nabla w \dx = 0.
$$
Using a H\"older inequality, we can show that
\begin{multline*}
\epsilon \left|\int_\Omega \alpha(u_\epsilon)|\nabla u_\epsilon|^{p-2} \nabla u_\epsilon \cdot \nabla w \dx\right| 
\\
\le \Vert \nabla w\Vert_\infty ~ \epsilon \left(\int_\Omega \alpha(u_\epsilon)\psi' (u_\epsilon) |\nabla u_\epsilon|^{p} \dx\right)^{\frac {p-1} p}\left(\int_\Omega  \frac {\alpha(u_\epsilon)}{\psi' (u_\epsilon)^{p-1}} \dx\right)^{\frac {1} p}
\\ \le  C  ~ \epsilon^{\frac {1} p} \left(\int_\Omega \frac {\alpha(u_\epsilon)}{\psi' (u_\epsilon)^{p-1}}  \dx\right)^{\frac {1} p}.
\end{multline*}
We now need to bound the final integral in the right hand side independently of $\epsilon$. Using \eqref{eq:estpsiprimeun}, it is sufficient to choose $\alpha$ such that
\[
\frac {\alpha(u_\epsilon)}{\psi' (u_\epsilon)^{p-1}}  \le  \frac 1 {\psi' (u_\epsilon)^{\hat q}},
\]
for some ${\hat q} < \frac {N}{N-2}$ (case $N\ge 3$). Since $\psi'(s)\in(0,1]$, taking  $\alpha(u_\epsilon) =  \psi' (u_\epsilon)^{r}$, for some $r>p-1-\frac {N}{N-2}$ satisfies this inequality. 

The question which then arises is the possibility to prove the existence of $u_\epsilon$. The existence proof in Section \ref{lin:sec:probreg} relies on the existence of $\tilde u_{\epsilon}\in W^{1,p}_0(\Omega)$ such that
\[
 \psi' (u_\epsilon)^{r}(u_{\epsilon}) |\nabla u_{\epsilon}|^{p-2}\nabla u_{\epsilon} = |\nabla \tilde u_{\epsilon}|^{p-2}\nabla \tilde u_{\epsilon},
\]
which yields the change of variable $\tilde u_{\epsilon} = \psi_r(u_{\epsilon})$ with $\psi_r(s) =\int_0^s 
\psi'(t)^{\frac{r}{p-1}} \dt$.  We then recover $u_{\epsilon}$ using the reciprocal function $(\psi_r)^{-1}$ of $\psi_r$, which requires that the domain of  $(\psi_r)^{-1}$  be equal to $\mathbb{R}$, and therefore that the image of $\psi_r$ be  $\mathbb{R}$. This does not hold for $r\ge p-1$: indeed, since for any $s\in\mathbb{R}$, $0<\psi'(s)\le 1$, there holds $0 < \psi'(s)^{\frac{r}{p-1}}\le \psi'(s)$ which yields $|\psi_r(s)|\le |\psi(s)|< 1$, for such $r$, the image of $\psi_r$ cannot be equal to $\mathbb{R}$.  Hence we have to choose $r\in (p-1-\frac {N}{N-2},p-1)$. We choose the value $r= p-2$, whose advantage is to be independent of $N$, and to lead to simpler expressions.

In consequence, the weighting function chosen in the remaining part of this paper is defined by 
\begin{equation*}
 \forall s\in \R,\ \alpha(s) =  (\psi'(s))^{p-2},
\end{equation*}
and we have $ \alpha(u_{\epsilon}) |\nabla u_{\epsilon}|^{p-2} =  |\nabla \psi(u_{\epsilon})|^{p-2}$.

In the next sections, we prove the existence of $u_\epsilon$, some
estimates on this function, and the convergence of 
$u_\epsilon$ to a weak or entropy
weak solution of the linear or quasilinear problems as $\epsilon\to 0$.

\section{Study of the regularised problem}\label{sec:regpb}

In the whole section, $\epsilon >0$ is given.

In view of Section \ref{sec:nonlinear}, we introduce a nonstrictly increasing function $\mu_\epsilon\in C(\R)$ such that $\mu_\epsilon(0) = 0$, which covers the case $\mu_\epsilon \equiv 0$ used in the linear case. As a consequence of Section \ref{sec:motiv}, we consider the following problem
\begin{equation}\label{lin:probconteps}
 \mu_\epsilon(u_{\epsilon})- \Div ( \Lambda \nabla u_{\epsilon} +  \epsilon  |\nabla \psi(u_{\epsilon})|^{p-2}\nabla u_{\epsilon}) = f,
\end{equation}
together with homogeneous Dirichlet boundary conditions
\begin{equation}\label{lin:probboundeps}
u_{\epsilon } = 0~\text{on}~\partial \Omega.
\end{equation}

\subsection{Existence of a solution to the regularised problem}\label{lin:sec:probreg}

\begin{lemma}[Existence of a weak solution to Problem \eqref{lin:probconteps}-\eqref{lin:probboundeps}]\label{lem:lin:existolreg}

There exists  a function $u_\epsilon$ such that 
\begin{multline}\label{lin:pbuepsilonreg}
u_\epsilon \in W^{1,p}_0(\Omega)
\hbox{ and for any }w \in W^{1,p}_0(\Omega),\\
\int_\Omega (\mu_\epsilon(u_\epsilon) w + \Lambda \nabla u_\epsilon \cdot \nabla w) \dx
+ \epsilon \int_\Omega |\nabla \psi(u_{\epsilon})|^{p-2} \nabla u_\epsilon \cdot \nabla w \dx 
= \int_\Omega w \d f.
\end{multline}
\end{lemma}

\begin{proof}

{\bf Step 1: change of variable.}

\medskip

We define the odd, strictly increasing diffeomorphism  $\psi_p~:~\mathbb{R}\to \mathbb{R}$ by
$$
\forall s\in\mathbb{R},\ \psi_p(s)=\int_0^s 
\psi'(t)^{\frac{p-2}{p-1}} \dt.
$$
Remarking that, for any $\tau\in(0,2)$, the minimum value of the function $s\mapsto (1+|s|)^{1+\tau}\psi'(s)$ is attained when $1+\ln(1+|s|) = \frac 2 \tau$, we get
\begin{equation}\label{eq:boundpsiprime}
\forall \tau\in (0,2),\  \forall s\in \mathbb{R},\ \frac {\tau^2} {4(1+|s|)^{1+\tau}}\le \psi' (s) =\frac 1 {(1+\ln(1+|s|))^2(1+|s|)} \le\frac 1 { 1+|s|}.
\end{equation}
This leads, for any $t\in (0,\frac 1 {p-1}]$, to
\[
 \forall s\in [0,+\infty),\ \frac {(1-t(p-1))^2}{4t (p-2)^2} ( (1+s)^{t}-1)\le \psi_p(s)\le  (p-1)((1+s)^{\frac 1 {p-1}}-1),
\]
which shows that the image $\psi_p$ is equal to $\mathbb{R}$.

In this proof, we are looking for the existence of $u_\epsilon$ solution to \eqref{lin:pbuepsilonreg}. For this purpose, we introduce the change of variable, which enables to solve by minimisation a $p-$Laplace problem without weight,
\begin{equation}\label{eq:chofvar}\tilde u_\epsilon=
\psi_p(u_\epsilon).
\end{equation}
This means that $u_\epsilon=\psi_p^{-1}(\tilde u_\epsilon)$, which can only be written using that the range of $\psi_p$ is equal to $\mathbb{R}$. It leads to $\nabla u_\epsilon = \nabla\psi_p^{-1}(\tilde u_\epsilon) = (\psi_p^{-1})'({\tilde u}_\epsilon)\nabla\tilde u_\epsilon$. Since $(\psi_p^{-1})'$ is continuous , for any ${\tilde u}_\epsilon\in W^{1,p}_0(\Omega)\subset L^\infty(\Omega)$, we get that $(\psi_p^{-1})'({\tilde u}_\epsilon)$ remains bounded, which implies that $u_\epsilon\in W^{1,p}_0(\Omega)$.

Besides, we can also write
\[
 |\nabla \psi(u_\epsilon)|^{p-2} \nabla u_\epsilon = (\psi'(u_\epsilon))^{p-2} |\nabla u_\epsilon|^{p-2} \nabla u_\epsilon = (\psi_p'(u_\epsilon))^{p-1} |\nabla u_\epsilon|^{p-2} \nabla u_\epsilon= |\nabla\tilde  u_\epsilon|^{p-2} \nabla\tilde  u_\epsilon,
\]
and $\mu_\epsilon(u_\epsilon)=\mu_\epsilon(\psi_p^{-1}({\tilde u}_\epsilon))$ and $\Lambda \nabla u_\epsilon = (\psi_p^{-1})'({\tilde u}_\epsilon)\Lambda \nabla {\tilde u}_\epsilon$.

Hence Problem \eqref{lin:pbuepsilonreg} is equivalent to find  $\tilde  u_\epsilon\in W^{1,p}_0(\Omega)$ such that
\begin{equation}\label{lin:pbepsilonbetap}
\int_\Omega (\mu_\epsilon(\psi_p^{-1}({\tilde u}_\epsilon)) w +
(\psi_p^{-1})'({\tilde u}_\epsilon) \Lambda 
\nabla {\tilde u}_\epsilon \cdot \nabla w) \dx
+ \epsilon \int_\Omega
|\nabla {\tilde u}_\epsilon|^{p-2} \nabla {\tilde u}_\epsilon \cdot \nabla w \dx 
= \int_\Omega w \d f, ~\text{for any}~w \in W^{1,p}_0(\Omega).
\end{equation}

\medskip

{\bf Step 2: existence of ${\tilde u}_\epsilon$.}

\medskip

In order to prove the existence of   $\tilde  u_\epsilon\in W^{1,p}_0(\Omega)$ such that \eqref{lin:pbepsilonbetap} holds, we remark that such a solution satisfies $\tilde  u_\epsilon = F(\tilde  u_\epsilon)$, where the mapping $F :  W^{1,p}_0(\Omega)
\longrightarrow W^{1,p}_0(\Omega)$ is such that, for any ${\tilde v} \in
W^{1,p}_0(\Omega)$, the element ${\tilde u} = F({\tilde v}) $ with  ${\tilde u}\in W^{1,p}_0(\Omega)$ and
\begin{equation}\label{Fpweaklin}
\int_\Omega( \mu_\epsilon(\psi_p^{-1}({\tilde v})) w +
(\psi_p^{-1})'({\tilde v})\Lambda 
\nabla {\tilde u} \cdot \nabla w )\dx
+ \epsilon \int_\Omega
|\nabla {\tilde u}|^{p-2} \nabla {\tilde u} \cdot \nabla w \dx 
= \int_\Omega w \d f, ~\text{for any}~w \in W^{1,p}_0(\Omega).
\end{equation}

We can then apply Lemma \ref{lem:minstcvxe}, letting $\sigma = \mu_\epsilon\circ\psi_p^{-1}$ and $\rho = (\psi_p^{-1})'$, which states that the mapping $F$ is well defined, continuous and compact.

\medskip

Let $t \in [0,1]$, and let ${\tilde u}\in W^{1,p}_0(\Omega)$ such that $t F({\tilde u}) = {\tilde u}$ (the existence of such $\tilde u$ is not yet proved). Let us prove that ${\tilde u}$ remains bounded. This is clear for $t=0$. Let us now assume that $t\in (0,1]$, and let $\tilde u$ satisfy $F({\tilde u}) = {\tilde u}/t$, which means that
\begin{multline*}
\int_\Omega ( t^{p-1} \mu_\epsilon(\psi_p^{-1}({\tilde u})) w+
t^{p-2} (\psi_p^{-1})'({\tilde u})\Lambda
\nabla {\tilde u} \cdot \nabla w )\dx
+ \epsilon \int_\Omega
|\nabla {\tilde u}|^{p-2} \nabla {\tilde u} \cdot \nabla w \dx \\
= t^{p-1} \int_\Omega w \d f, ~\text{for any}~w \in W^{1,p}_0(\Omega).
\end{multline*}
Letting $w={\tilde u}$,  we get, since $\mu_\epsilon(\psi_p^{-1}(\tilde u))\tilde u \ge 0$ and $(\psi_p^{-1})'({\tilde u})\Lambda
\nabla {\tilde u} \cdot \nabla \tilde u\ge 0 $, that
\[
  \epsilon\Vert \nabla {\tilde u}\Vert_p^{p-1} \le C_{\rm sob}^{(p,\infty)} \Vert f\Vert_{M(\Omega)},
\]
which shows that $u$ is bounded independently of $t$.

Hence the function $F$, which is continuous and compact from $W^{1,p}_0(\Omega)$ to $W^{1,p}_0(\Omega)$, is such that there exists $C$ such that, for any $t\in [0,1]$ and for any solution ${\tilde u}$ to $tF({\tilde u}) = {\tilde u}$, then $\Vert \nabla {\tilde u}\Vert_p \le C$; we can then apply Schaefer's fixed point theorem \cite{schaefer1955uber} (which is deduced from Leray-Schauder topological degree theory), which proves that there exists ${\tilde u}\in W^{1,p}_0(\Omega)$ such that $F({\tilde u}) = {\tilde u}$.
\end{proof}
\begin{remark}
 In the case where $\mu_\epsilon =0$, it is possible to get directly from \eqref{Fpweaklin} the existence of a fix-point, by applying Leray-Schauder fix-point theorem (in this case, the norm of $\tilde u$ is bounded independently of $\tilde v$).
\end{remark}

\begin{lemma}[A continuous compact operator]\label{lem:minstcvxe} 

Let $\sigma\in C(\mathbb{R})$ and $\rho\in C(\mathbb{R})$  be given, such that $\rho(s)\ge 0$ for all $s\in \mathbb{R}$. Then for all $\tilde v\in  W^{1,p}_0(\Omega)$, there exists one and only one  function $\tilde u$ such that 
\begin{multline}\label{eq:ministcvxe}
\tilde u \in W^{1,p}_0(\Omega)
\hbox{ and for any }w \in W^{1,p}_0(\Omega),\\
\int_\Omega \sigma(\tilde v) w \dx+\int_\Omega \rho(\tilde v) \Lambda \nabla \tilde u \cdot \nabla w \dx
+ \epsilon \int_\Omega |\nabla \tilde u|^{p-2} \nabla \tilde u \cdot \nabla w \dx 
= \int_\Omega w \d f.
\end{multline}
Moreover, denoting by $F$ the mapping $\tilde v\mapsto \tilde u$, then $F$ is continuous and compact.
\end{lemma}
\begin{proof}
{\bf Step 1: existence of $\tilde u$ solution to \eqref{eq:ministcvxe}.}

\medskip

Let us define the function $ \mathcal{I}_{\tilde v} : W_0^{1,p}(\Omega) \to \R$ defined for any $w \in W^{1,p}_0(\Omega) \subset L^\infty(\Omega)\cap H^1_0(\Omega)$ by
$$
\mathcal{I}_{\tilde v}(w)
= \int_\Omega \sigma(\tilde v) w \dx + \frac{\epsilon}{p} \int_\Omega | \nabla w |^{p} \dx
+ \frac{1}{2}\int_\Omega \rho(\tilde v) \Lambda 
 \nabla w \cdot \nabla w \dx - \int_\Omega w \d f.
$$
We have, for any $\alpha>0$, that 
\begin{multline*}
 \int_\Omega w \d f - \int_\Omega \sigma(\tilde v) w \dx\le \Vert w\Vert_\infty(\Vert f\Vert_{M(\Omega)}+\Vert \sigma(\tilde v)\Vert_1)\le C_{\rm sob}^{(p,\infty)}\Vert \nabla w\Vert_p(\Vert f\Vert_{M(\Omega)}+\Vert \sigma(\tilde v)\Vert_1)
 \\ \le \alpha \frac {\Vert \nabla w\Vert_p^p}{p} + \frac {1}{\alpha} \frac {(\Vert f\Vert_{M(\Omega)}+\Vert \sigma(\tilde v)\Vert_1)^{p'}}{p'}.
\end{multline*}
Since $\rho(\tilde v)\ge 0$, choosing $\alpha=\frac {\epsilon} 2$ shows that there exists $c_2\ge 0$ such that
\[
 \forall w\in W^{1,p}_0(\Omega),\ \mathcal{I}_{\tilde v}(w)\ge \frac {\epsilon} 2\Vert \nabla w\Vert_p^p - c_2.
\]
 We then get that $\mathcal{I}_{\tilde v}(w)$ is bounded by below independently of $w$, and that $\mathcal{I}_{\tilde v}(w)\to +\infty$ if $\Vert \nabla w\Vert_p\to +\infty$. Therefore there exists a bounded minimizing sequence $(w_n)_{n\in\mathbb{N}}$. Hence there exists a subsequence, again denoted by $(w_n)_{n\in\mathbb{N}}$, which is weakly converging to some $\tilde u\in W_0^{1,p}(\Omega)$ (and therefore also weakly converging  in $H^1_0(\Omega)$). Using that the norm function is weakly lower semicontinuous and the positivity of $\int_\Omega \rho(\tilde v)  \Lambda 
 \nabla (w_n - \tilde u) \cdot \nabla (w_n - \tilde u) \dx$, we get that
\[
 \Vert \nabla \tilde u\Vert_p \le \liminf_{n\to +\infty}\Vert \nabla w_n\Vert_p \hbox{ and }\int_\Omega \rho(\tilde v)  \Lambda 
 \nabla \tilde u \cdot \nabla \tilde u \dx \le \liminf_{n\to +\infty} \int_\Omega \rho(\tilde v) \Lambda 
 \nabla w_n \cdot \nabla w_n \dx.
\]
This implies
\[
 \mathcal{I}_{\tilde v}(\tilde u)\le  \liminf_{n\to +\infty}\mathcal{I}_{\tilde v}(w_n),
\]
which proves that $\tilde u$ is a minimizer of $\mathcal{I}_{\tilde v}(w)$. Then, for any $w \in W^{1,p}_0(\Omega)$, the function  defined for all $t\in\mathbb{R}$ by $\mathcal{I}_{\tilde v}(\tilde u + tw)$, admits a minimum in $t=0$.

Computing the derivatives of the function $g(t) = |x+ty|^p$ for $x,y\in\mathbb{R}^N$ and using that $p>2$, we get that
\[
 \forall t\in [0,1],\ \forall x,y\in\mathbb{R}^N,\ |g''(t)|\le p(p-1) |y|^2(|x| + |y|)^{p-2}.
\]
This proves the right inequality in 
\[
\forall x,y\in\mathbb{R}^N,\ 0 \le  |x+y|^p - |x|^p - p |x|^{p-2} x\cdot y \le \frac {p(p-1)} 2 |y|^2(|x| + |y|)^{p-2},
\]
and therefore, in addition to $|x| + |y|\le 2 \max(|x|,|y|)$ that
\[
\forall t\in [-1,1],\ \forall x,y\in\mathbb{R}^N,\ 0 \le  |x+t y|^p - |x|^p - t p |x|^{p-2} x\cdot y \le t^2 p(p-1) 2^{p-3} \max(|x|^p,|y|^p).
\]
In addition to 
\[
 \forall t\in [-1,1],\ \forall x,y\in\mathbb{R}^N, \ \Lambda (x+ty)\cdot (x+ty) - \Lambda x \cdot x - 2 t \Lambda x \cdot y = t^2 \Lambda y \cdot y,
\]
we get that the expression defined for $t\in [-1,1]\setminus\{0\}$ by
\[
A(t) =  \frac {\mathcal{I}_{\tilde v}(\tilde u + tw) - \mathcal{I}_{\tilde v}(\tilde u )} t -\Big( \int_\Omega \sigma(\tilde v) w \dx+\int_\Omega \rho(\tilde v) \Lambda \nabla \tilde u \cdot \nabla w \dx
+ \epsilon \int_\Omega |\nabla \tilde u|^{p-2} \nabla\tilde  u \cdot \nabla w \dx 
- \int_\Omega w \d f\Big),
\]
satisfies
\[
 |A(t)| \le |t| \ \Big( \frac{\epsilon}{p} p(p-1) 2^{p-3} (\Vert \tilde u\Vert_p^p + \Vert w\Vert_p^p) + \overline{\lambda} \Vert\rho(\tilde v)\Vert_\infty \Vert w\Vert_2^2 \Big),
\]
and therefore $\lim_{t\to 0}A(t) = 0$.
Letting $t\to 0$ with $t>0$ and $t<0$ successively, observing that $\frac {\mathcal{I}_{\tilde v}(\tilde u + tw) - \mathcal{I}_{\tilde v}(\tilde u )}{t}$ has the sign of $t$ since $\mathcal{I}_{\tilde v}(\tilde u )$ minimizes $\mathcal{I}_{\tilde v}$, we obtain that
\[
 0 = \int_\Omega \sigma(\tilde v) w \dx+\int_\Omega \rho(\tilde v) \Lambda \nabla \tilde u \cdot \nabla w \dx
+ \epsilon \int_\Omega |\nabla \tilde u|^{p-2} \nabla\tilde  u \cdot \nabla w \dx 
- \int_\Omega w \d f.
\]
Therefore \eqref{eq:ministcvxe} holds for the minimizer $\tilde u$ of $\mathcal{I}_{v}$, which shows the existence of at least one solution to \eqref{eq:ministcvxe}.

\medskip

{\bf Step 2: uniqueness.}

\medskip

For ${\tilde v}_1,{\tilde v}_2\in  W^{1,p}_0(\Omega)$, let ${\tilde u}_1,{\tilde u}_2\in  W^{1,p}_0(\Omega)$ be respective solutions to \eqref{eq:ministcvxe}.
We get, for any  $w \in W^{1,p}_0(\Omega)$,
\begin{equation}\label{eq:tildeuun}
\int_\Omega \Lambda
\rho({\tilde v}_1)
 \nabla  {\tilde u}_1  \cdot \nabla w \dx
+ \epsilon \int_\Omega | \nabla  {\tilde u}_1 |^{p-2} \nabla  {\tilde u}_1  \cdot \nabla w \dx
= \int_\Omega w \d f -  \int_\Omega \sigma( {\tilde v}_1) w \dx,
\end{equation}

and 

\begin{multline*}
\int_\Omega \Lambda
\rho({\tilde v}_1)
 \nabla {\tilde u}_2  \cdot \nabla w \dx
+ \epsilon \int_\Omega | \nabla {\tilde u}_2 |^{p-2} \nabla {\tilde u}_2  \cdot \nabla w \dx
= \int_\Omega w \d f -  \int_\Omega \sigma({\tilde v}_2) w \dx
\\
+ \int_\Omega \Lambda (
\rho({\tilde v}_1)
-
\rho({\tilde v}_2)
) \nabla {\tilde u}_2  \cdot \nabla w \dx.
\end{multline*}
Letting $w = {\tilde u}_1 - {\tilde u}_2$ in the first one and $w = {\tilde u}_2 - {\tilde u}_1$ in the second one, adding both equations and using the inequality 
\cite[Lemma 2.40]{gdm}, which holds since $p\ge 2$,
\[
 \forall x,y\in\R^N,\,|x - y|^p \le 2^{p-1}(|x|^{p-2}x -|y|^{p-2}y)(x - y),
\]
we obtain

\begin{multline}\label{eq:fcpct}
\int_\Omega \Lambda
\rho({\tilde v}_1)
 \nabla ({\tilde u}_2- {\tilde u}_1)  \cdot \nabla ({\tilde u}_2- {\tilde u}_1) \dx
+ \epsilon 2^{1-p}\int_\Omega | \nabla ({\tilde u}_2- {\tilde u}_1) |^{p} \dx
\le   \int_\Omega (\sigma({\tilde v}_1) - \sigma({\tilde v}_2))({\tilde u}_2 - {\tilde u}_1 )\dx
\\
+ \int_\Omega \Lambda (
\rho({\tilde v}_1)
-
\rho({\tilde v}_2)
) \nabla  {\tilde u}_2  \cdot \nabla ({\tilde u}_2 - {\tilde u}_1) \dx.
\end{multline}
The above inequality shows that, for ${\tilde v}_1 = {\tilde v}_2$, then ${\tilde u}_1 = {\tilde u}_2$ (and therefore that \eqref{eq:tildeuun} characterises the minimizer of $\mathcal{I}_{{\tilde v}_1}$). Therefore the mapping $F~:~\tilde v\mapsto\tilde u$ unique solution of \eqref{eq:ministcvxe} is well defined.

\medskip

{\bf Step 3: continuity and compactness of $F$.}

\medskip

Letting $w={\tilde u}$ in \eqref{eq:ministcvxe}, we get that
\begin{equation}\label{eq:fcpctbound}
 \epsilon\Vert \nabla {\tilde u}\Vert_p^p \le C_{\rm sob}^{(p,\infty)} (\Vert f\Vert_{M(\Omega)}+\Vert \sigma({\tilde v})\Vert_1).
\end{equation}
We then obtain $\Vert \tilde u\Vert_2$ and $\Vert \nabla \tilde u\Vert_2$ are increasingly depending of
$ (\Vert f\Vert_{M(\Omega)}+\Vert \sigma({\tilde v})\Vert_1)$.

Let $({\tilde v}_n)_{n\in\mathbb{N}}$ be  a bounded sequence of $W^{1,p}_0(\Omega)$. We extract a subsequence, again denoted $({\tilde v}_n)_{n\in\mathbb{N}}$ such that  $\tilde v_n$ weakly converges to some ${\tilde v}\in W^{1,p}_0(\Omega)$ and strongly in $L^\infty(\Omega)$. We then have the convergence in $L^\infty(\Omega)$ of $\sigma(\tilde v_n)$ and  $\rho(\tilde v_n)$ respectively to $\sigma(\tilde v)$ and $\rho(\tilde v)$.

\medskip

Inequality \eqref{eq:fcpct} in which we let $\tilde v_1 = \tilde v_n$, $\tilde v_2 = \tilde v$, $\tilde u_1 = F( \tilde v_n)$, $\tilde u_2 = F(\tilde v)$, we get
\begin{multline*}
\epsilon 2^{1-p}\int_\Omega | \nabla (F({\tilde v})- F({\tilde v}_n)) |^{p} \dx
\le   \int_\Omega (\sigma({\tilde v}_n) - \sigma({\tilde v}))(F(\tilde v) - F( \tilde v_n) )\dx \\
+ \int_\Omega \Lambda (
\rho({\tilde v}_n)
-
\rho({\tilde v})
) \nabla  F(\tilde v) \cdot \nabla (F(\tilde v) - F( \tilde v_n)) \dx.
\end{multline*}
Notice that \eqref{eq:fcpctbound} implies that $F(\tilde v) - F( \tilde v_n)$ remains bounded in $L^1(\Omega)$ as well as $\nabla  F(\tilde v) \cdot \nabla (F(\tilde v) - F( \tilde v_n))$. Therefore, using the above convergences in $L^\infty(\Omega)$, we get that the right hand side of the above inequality tends to 0 and therefore $F({\tilde v}_n)$ tends to $F({\tilde v})$ in $W^{1,p}_0(\Omega)$. This shows that $F$ is compact and at the same time  that it is continuous.
\end{proof}

\subsection{Estimates on the solution of the regularised problem}\label{lin:sec:estimates}
We define the following function, which is an odd, strictly increasing diffeomorphism from $\mathbb{R}$ to $\mathbb{R}$:
\begin{equation} \label{lin:eq:deftildechi}
    \tilde\chi : 
		s \longmapsto  \int_0^s \sqrt{ \psi'(t)} \d\,t.
  \end{equation}

\begin{lemma}\label{lin:lem:estimep}

Let $u_\epsilon$ be
given such that \eqref{lin:pbuepsilonreg} holds.
Then there exists $\ctel{lin:cste:l2}$ only depending on $\underline{\lambda}$ and $\Vert f\Vert_{M(\Omega)}$ such that 
 \begin{equation}\label{eq:lin:majchitilde}
 \int_\Omega 
\psi'(u_\epsilon)
|\nabla u_\epsilon|^2 
\dx =  {\Vert\nabla \tilde\chi(u_\epsilon)\Vert_2}^2
\le \cter{lin:cste:l2}
 \end{equation}
and there holds 
\begin{equation}\label{eq:lin:weakbound}
\epsilon \int_\Omega (\psi'(u_\epsilon))^{p-1}
| \nabla u_\epsilon|^{p} \dx 
 \le \Vert f\Vert_{M(\Omega)}.
\end{equation}
\end{lemma}
\begin{proof}
Taking $w = \psi (u_\epsilon) $ in ($\ref{lin:pbuepsilonreg}$), we obtain
\[
\int_\Omega (\mu_\epsilon(u_\epsilon)\psi (u_\epsilon) + \Lambda \nabla u_\epsilon \cdot \nabla \psi (u_\epsilon)) \dx 
+ \epsilon \int_\Omega (\psi'(u_\epsilon))^{p-2}
| \nabla u_\epsilon|^{p-2}
\nabla u_\epsilon \cdot \nabla \psi (u_\epsilon) \dx 
= \int_\Omega \psi (u_\epsilon) \d f,
\]
which provides, since $\mu_\epsilon(s)\psi (s)\ge 0$,
$$
\int_\Omega \psi'(u_\epsilon)\Lambda \nabla u_\epsilon \cdot \nabla u_\epsilon \dx
+ \epsilon \int_\Omega (\psi'(u_\epsilon))^{p-1}
| \nabla u_\epsilon|^{p}
\dx 
\le \int_\Omega \psi (u_\epsilon) \d f.
$$
Using the fact that $ \| \psi (u_\epsilon)  \|_{\infty} \le 1$ we then obtain \eqref{eq:lin:majchitilde} and \eqref{eq:lin:weakbound}.
\end{proof}

\begin{lemma}\label{lin:lem:estlrep}
Under the assumptions of Lemma \ref{lin:lem:estimep}, for any $ 1 < q < \frac{N}{N-1}$, there exists $\ctel{lin:cste:lr}$ only depending on $q$, $N$, $\underline{\lambda}$ and $\Vert f\Vert_{M(\Omega)}$ such that
 \begin{equation}\label{lin:eq:estimmupsi}
 \Vert\nabla u_\epsilon\Vert_q \dx
\le \cter{lin:cste:lr},
\end{equation}
and, letting ${\hat q} = q/(2-q)$ (then ${\hat q}\in (1,+\infty)$ if $N=2$ and  ${\hat q}\in (1, N/(N-2))$ if $N\ge 3$),
 \begin{equation}\label{lin:eq:estimpsi}
 \Vert u_\epsilon\Vert_{{\hat q}} \le \cter{lin:cste:lr} \hbox{ and }\Vert 1/\psi' ( u_\epsilon)\Vert_{{\hat q}} \le  \cter{lin:cste:lr}.
\end{equation}
\end{lemma}
\begin{remark}
It suffices to apply \cite[Lemma 2.2]{BGV} for getting the proof of \eqref{lin:eq:estimmupsi} and the left part of \eqref{lin:eq:estimpsi}, remarking that \eqref{eq:boundpsiprime} provides  \cite[(2.16)]{BGV} for any $m\in(0,1)$.
In the next proof, we are essentially using the ideas issued from  \cite{BG}-\cite{BGV} for proving \eqref{lin:eq:estimmupsi} and the left part of \eqref{lin:eq:estimpsi}, with a slightly different way for applying the Sobolev inequalities. Another small difference is the use of the function $\psi$ instead of the function $s\mapsto (1-(1+|s|)^{-m}){\rm sign}(s)$.
\end{remark}

\begin{proof}
Using Hölder's inequality with conjugate exponents $ \frac{2}{q}>1 $ and $\frac{2}{2-q} $ and owing to \eqref{eq:lin:majchitilde} in Lemma \ref{lin:lem:estimep}, we obtain
\begin{multline*}
\int_\Omega |\nabla u_\epsilon|^q \dx
=
\int_\Omega |\nabla u_\epsilon|^q 
\left(
\frac
{\psi'(u_\epsilon)}
{\psi'(u_\epsilon)}
\right)^{q/2} \dx
\\
\leq
\left(\int_\Omega\psi'(u_\epsilon)|\nabla u_\epsilon|^2 \dx\right)^{q/2}
\left(\int_\Omega\frac 1 {(\psi'(u_\epsilon))^{q/(2-q)}}\dx\right)^{(2-q)/2} \\
\le \left(\cter{lin:cste:l2}\right)^{q/2}
\left(\int_\Omega \frac 1 {(\psi'(u_\epsilon))^{q/(2-q)}}\dx\right)^{(2-q)/2}.
\end{multline*}
Our aim is now to bound the $L^{q/(2-q)}$ norm of  $1/\psi' ( u_\epsilon)$, using Sobolev inequalities and the $L^2$ bound \eqref{eq:lin:majchitilde} on  $\nabla \tilde\chi(u_\epsilon) $. For this purpose, we compare, for any $s\in \mathbb{R}$, the expression  $1/\psi' ( s)$ with powers of  $\tilde\chi(s) $.
Let us recall that \eqref{eq:boundpsiprime} states that the main part of $1/\psi' (s)$ is $ 1+|s|$, up to an arbitrary small exponent $\tau$. In particular, the left inequality of \eqref{eq:boundpsiprime} provides, for any $\tau\in (0,2)$ and $s\ge 0$,
\begin{equation}\label{eq:bounduspsiprime}
\frac 1 {\psi'(s)} \le \frac{4(1+s)^{1+\tau}} {\tau^2}.
\end{equation}
Recall that $\tilde\chi(s)$ is a primitive of $\sqrt{\psi'(s)}$, and is therefore expected to behave, up to an arbitrary small exponent, as $\sqrt{1+|s|}$. Indeed, for any $\tau\in (0,2)$ and $s>0$, taking the square root of  the left part of \eqref{eq:boundpsiprime} gives
\[
 \frac {\tau} {2(1+s)^{(1+\tau)/2}}\le \sqrt{\psi' (s)}.
\]
Considering only $\tau\in (0,1)$ and integrating the preceding relation between $0$ and $s\ge 0$  provides
\begin{equation}\label{eq:boundchitilde}
\tau ((1+s)^{\frac {1 - \tau} 2}-1)\le  \tilde\chi(s).
\end{equation}
Eliminating $s$ between \eqref{eq:bounduspsiprime} and \eqref{eq:boundchitilde} yields, for any $\tau\in (0,1)$,
\begin{equation}\label{lin:eq:psiprimechi}
\forall s\in\mathbb{R},\ \frac 1 {\psi' (s)}\le \frac {4} {\tau^2} \big( \frac 1 \tau |\tilde\chi(s)|+1\big)^{2\frac{1+\tau}{1-\tau}}.
\end{equation}
We thus obtain that, defining $\rho(\tau) = 2 \frac {(1+\tau)q} {(1-\tau)(2-q)}$, there exist $\ctel{lin:cte:alphataur}^{(q,\tau)}$ and $\ctel{lin:cte:betataur}^{(q,\tau)}$ such that
$$
\int_\Omega
\frac 1 {(\psi'(u_\epsilon))^{q/(2-q)}}\dx
\leq
\cter{lin:cte:alphataur}^{(q,\tau)}
\int_\Omega
|\tilde\chi(u_\epsilon)|^{\rho(\tau)}
\dx
+
\cter{lin:cte:betataur}^{(q,\tau)}
|\Omega|
$$
\begin{itemize}
\item In the case $N=2$, let us define $\tau = \frac 1 2$. Then the Sobolev inequality \eqref{eq:sobolev} provides
\[
 \Vert \tilde\chi(u_\epsilon)\Vert_{\rho(\tau)} \le C_{\rm sob}^{(2,\rho(\tau))} \| \nabla \tilde\chi(u_\epsilon)\|_{2}.
\]

\item In the case $N>2$, let us select $\tau\in (0,1)$ such that
\begin{equation}
\frac {1+\tau} {1-\tau}  = \frac {N} {N-2} \frac {2-q} q.\label{lin:eq:defeps}
\end{equation}
Indeed, since $q\in  [1,N/(N-1))$ implies $(2-q)/q \in  ((N-2)/N,1]$,
the quantity $a_q$ such that $a_q =\frac {N} {N-2}
\frac {2-q} q$ is such that $1<a_q$. Since \eqref{lin:eq:defeps} leads to $\tau = \frac {a_q - 1}{a_q + 1}$, we get that $\tau\in (0, 1)$ and that $\rho(\tau) = \frac{2N}{N-2} $, and
\eqref{eq:sobolev} holds for any such $\rho(\tau)$.
It leads to 
\[
 \Vert \tilde\chi(u_\epsilon)\Vert_{\rho(\tau)} \le C_{\rm sob}^{(2,\rho(\tau))} \| \nabla \tilde\chi(u_\epsilon)\|_{2}.
\]
\end{itemize}
Gathering the preceding inequalities and applying again \eqref{eq:lin:majchitilde} in Lemma \ref{lin:lem:estimep} provide
\[
\int_\Omega |\nabla u_\epsilon|^q \dx
\leq
\left(\cter{lin:cste:l2}\right)^{q/2}
\left(\cter{lin:cte:alphataur}^{(q,\tau)}
\big( C_{\rm sob}^{(2,\rho(\tau))} (\cter{lin:cste:l2})^{\frac 1 2} \big)^{ \rho(\tau)}
+
\cter{lin:cte:betataur}^{(q,\tau)}
|\Omega|
\right)^{(2-q)/2},
\]
which provides \eqref{lin:eq:estimmupsi}.
A Sobolev inequality then yields the left inequality of \eqref{lin:eq:estimpsi}. The right one is then a consequence of \eqref{eq:bounduspsiprime}
and of the choice of $\tau$ such that ${\hat q}(1+\tau) < N/(N-2)$ if $N>2$.
\end{proof}
\begin{lemma}
Under the assumptions of Lemma \ref{lin:lem:estimep}, there exists $\ctel{lin:cste:tk}$ only depending on $\underline{\lambda}$, $k$ and $ \| f\|_{M(\Omega)}$ such that
 \[
 \| \nabla T_k( u_\epsilon)\|_{2} \le \cter{lin:cste:tk},
\]
where  $T_k$ is the truncation function defined by $T_k(s) = \min(|s|,k){\rm sign}(s)$ for all $s\in\R$ (where ${\rm sign}(s)=1$ if $s\ge 0$ and $-1$ if $s < 0$).
\end{lemma}
\begin{proof}
Using that $T_k'(s)=1$ for $|s|\le k$ and  $T_k'(s)=0$ for $|s|>k$, as well as $|\nabla u_{\epsilon} |^2 =\frac {1} {\psi' (u_{\epsilon})} |\nabla\tilde\chi( u_{\epsilon}) |^2 $,
we have that
\[
 \int_\Omega |\nabla T_k  u_{\epsilon} |^2 \dx =  \int_{|u_\epsilon|\le k}   |\nabla u_{\epsilon} |^2 \dx =  \int_{|u_\epsilon|\le k}  \frac {1} {\psi' (u_{\epsilon})} |\tilde\chi( u_{\epsilon}) |^2 \le \frac {1} {\psi' (k)}  \int_\Omega |\nabla\tilde\chi( u_{\epsilon}) |^2 \dx.
\]
We conclude, using  \eqref{eq:lin:majchitilde} in Lemma \ref{lin:lem:estimep}.
\end{proof}

\section{Convergence of the regularized problem to the linear problem}\label{sec:linear}

In this section, we consider the case $\mu_\epsilon\equiv 0$ in Section \ref{sec:regpb}, and we study how the resulting Problem \eqref{lin:probconteps}-\eqref{lin:probboundeps} is an approximation of Problem \eqref{eq:lin:probcont}-\eqref{eq:probbound}.

\subsection{Convergence to a weak solution}\label{lin:sec:convergenceweak}

Let us provide a weak sense for a solution of Problem \eqref{eq:lin:probcont}-\eqref{eq:probbound}.
\begin{definition}\label{def:weaksol}
We define the space $\mathcal{S}_N(\Omega)$ containing any solution and the space $\mathcal{T}_N(\Omega)$ containing the test functions by
\begin{equation}
\mathcal{S}_N(\Omega) = \bigcap_{q\in(1,\frac N {N-1})} W^{1,q}_0(\Omega)\hbox{ and }\mathcal{T}_N(\Omega)  = \bigcup_{r\in(N,+\infty)} W^{1,r}_0(\Omega) \subset C(\overline{\Omega}),\label{u:eq:defS7}
\end{equation}
We say that  a measurable function $u $ is a weak solution to Problem \eqref{eq:lin:probcont}-\eqref{eq:probbound}  if
\begin{equation}
u  \in S_N(\Omega)\hbox{ and }\int_\Omega \Lambda\nabla u\cdot\nabla w \dx = \int_\Omega w \d f,~\text{for any}~w \in \mathcal{T}_N(\Omega).
\label{eq:lin:pbcons}\end{equation}
\end{definition}
Let us observe that $u \in S_N(\Omega)$ implies that $u \in L^{\hat q}(\Omega)$ for any $\hat q\in(1,\frac N {N-2})$ if $N>2$ and for  any $\hat q\in(1,+\infty)$ if $N=2$.  Note that all $w\in \mathcal{T}_N(\Omega)$ is an element of $C(\overline{\Omega})$.

We can now state a result of existence, obtained by convergence of a solution to the regularised problem to a solution of the continuous problem, which holds owing to the choice done in Section \ref{sec:motiv} for the weight in the vanishing $p-$Laplace term.
\begin{lemma} \label{lin:u:lem:weakcv} Let  $(\epsilon_n)_{n \in \mathbb{N}} $ be a sequence of  positive numbers which converges to zero, and let  $u_{n}\in W^{1,p}_0(\Omega)$ be such that \eqref{lin:pbuepsilonreg} holds with $\epsilon = \epsilon_n$.

Then there exist a subsequence of  $(\epsilon_n,u_n)_{n\in \mathbb{N}} $, again denoted $(\epsilon_n,u_n)_{n\in \mathbb{N}} $, and $u\in \mathcal{S}_N(\Omega)$, such that the sequence $(u_n)_{ n \in \mathbb{N}}$ converges to $u \in S_N(\Omega)$ weakly in $W^{1,q}_0(\Omega)$ for any $q\in (1, N/(N-1))$, strongly  in $L^{\hat q}(\Omega)$  for all $\hat q\in [1,+\infty)$ if $N=2$ and for all $\hat q\in [1, N/(N-2))$ if $N\ge 3$ and almost everywhere in $\Omega$.

Moreover, $u$ is a weak solution in the sense of Definition \ref{def:weaksol}, and, for any $ 1 < q < \frac{N}{N-1}$, there exists $\ctel{lin:cste:lru}$ only depending on $q$, $N$, $\underline{\lambda}$ and $\Vert f\Vert_{M(\Omega)}$ such that
 \[
 \| \nabla u\|_{q} \le \cter{lin:cste:lru} .
\]
\end{lemma}
\begin{proof}
Using Lemma \ref{lin:lem:estlrep}, there exists a subsequence, again denoted $(\epsilon_n,u_n)_{n\in \mathbb{N}} $, such that the sequence $(u_n)_{ n \in \mathbb{N}}$ converges to a function $u \in S_N(\Omega)$ weakly in $W^{1,q}_0(\Omega)$ for any $q\in (1, N/(N-1))$, strongly  in $L^{\hat q}(\Omega)$  for all $\hat q\in [1,+\infty)$ if $N=2$ and for all $\hat q\in [1, N/(N-2))$ if $N\ge 3$ and almost everywhere in $\Omega$.

Let $ \phi \in  C^\infty_c(\Omega)$ and $n\in \mathbb{N}$. We have
\begin{equation}\label{lin:wsstep1}
\int_\Omega \Lambda \nabla u_n \cdot \nabla \phi \dx  + \epsilon_{n} \int_\Omega | \nabla  \psi  (u_n) |^{p-2} \nabla  u_n \cdot \nabla \phi \dx  =\int_\Omega \phi \d f.
\end{equation}
The weak convergence in $W^{1,q}_0(\Omega)$ for $q\in (1, N/(N-1))$ of the sequence $(u_n)_{n \in \mathbb{N}}$ to $u$ gives
$$
\lim_{n \to \infty} \int_\Omega \Lambda \nabla u_n\cdot \nabla \phi \dx = \int_\Omega \Lambda \nabla u \cdot \nabla \phi \dx.
$$ 
Using Hölder's inequality, we have
\begin{multline*}
\Big| \epsilon_{n} \int_\Omega| \nabla  \psi  (u_n) |^{p-2} \nabla  u_n \cdot \nabla \phi \dx  \Big| \le \epsilon_{n} \| \nabla \phi \|_\infty \int_\Omega(\psi' (u_n))^{p-2}| \nabla  u_n |^{p-1}  \dx
 \\
\le \epsilon_{n} \| \nabla \phi \|_\infty \Big( \int_\Omega \frac {1}{\psi'  (u_n)} \dx \Big)^{ \frac{1}{p}} \Big( \int_\Omega (\psi' (u_n))^{p-1}| \nabla  u_n |^{p}  \dx \Big)^{\frac{p-1}{p}}.
\end{multline*}
Applying  Lemma \ref{lin:lem:estimep} and Lemma \ref{lin:lem:estlrep}, we therefore get
\[
\Big| \epsilon_{n} \int_\Omega| \nabla  \psi  (u_n) |^{p-2} \nabla  u_n \cdot \nabla \phi \dx  \Big| 
\le (\epsilon_{n})^{1/p} \| \nabla \phi \|_\infty \Big( \cter{lin:cste:lr} \Big)^{\frac{1}{p}} \Big( \Vert f\Vert_{M(\Omega)}  \Big)^{\frac{p-1}{p}}.
\]
This shows that, letting $n\to\infty$ in \eqref{lin:wsstep1} with $n\in \mathbb{N}$, we obtain that $u$ satisfies \eqref{eq:lin:pbcons} for any $ \phi \in  C^\infty_c(\Omega)$. We then conclude the proof of the lemma by a density argument.
\end{proof}

\subsection{Convergence to the entropy weak solution}\label{lin:sec:convergenceent}

As announced by the introduction, we now use the definition of the entropy solution given in \cite{ben1995theory} (which is shown to be unique). In the whole section, we consider $f\in L^1(\Omega)$.

\begin{definition}\label{def:lin:entsol}
We define the entropy solution to Problem \eqref{eq:lin:probcont}-\eqref{eq:probbound} as the  measurable function $u$ such that
\begin{enumerate}
 \item $u \in S_N(\Omega)$ and, for all $k>0$, $T_k(u)\in H^1_0(\Omega)$, where we recall that  $T_k$ is the truncation function defined by $T_k(s) = \min(|s|,k){\rm sign}(s)$ for all $s\in\R$ (where ${\rm sign}(s)=1$ if $s\ge 0$ and $-1$ if $s < 0$),
 \item the following holds
 \begin{equation}\label{eq:lin:entropyineq}
\int_{ \Omega} \Lambda \nabla u \cdot \nabla T_k( u-\phi)  \dx  \le 
\int_{\Omega} f\,  T_k(u -\phi) \dx,
\end{equation}
for any $ \phi \in C_c^\infty(\Omega)$ and for any $k >0$.
\end{enumerate}
\end{definition}
\begin{remark}\label{rem:sensok}
Let $ k >0$ and $ \phi \in C_c^\infty(\Omega)$. Using the fact $ \{ | u - \phi | < k \} $ is a subset of $\{ |u | < h := k + \| \phi \|_\infty  \} $ we obtain 
$$
| \nabla T_k( u - \phi) | \le ( | \nabla u | + | \nabla \phi | ) 1_{ | u |< h} \le | \nabla T_h( u) |+ | \nabla \phi |.
$$
 The assumption that $T_h(u) \in H_0^1(\Omega) $ for any $h>0$ implies that $T_k(u - \phi ) \in H_0^1(\Omega) $ for any $k > 0$ and for any $ \phi \in C_c^\infty(\Omega)$. 
\end{remark}

Let us now turn to the convergence to the entropy solution.
\begin{lemma}\label{lin:lem:uepstendtoentsol}
Let $u$ be given by Lemma \ref{lin:u:lem:weakcv}. Then, for all $k>0$, there exists $\ctel{lin:cste:tku}$ only depending on $\underline{\lambda}$, $k$, and $ \| f\|_{1}$ such that
\[
 \| \nabla T_k( u)\|_{2} \le \cter{lin:cste:tku}.
\]
Moreover, $u$ is the entropy weak solution in the sense of Definition \ref{def:lin:entsol}.
\end{lemma}
\begin{proof}   The sequence $(T_k(u_n))_{n \ge 0}$ weakly converges to $T_k( u)$ in $H^1_0(\Omega)$ for any $k >0$.
Let $ \phi \in C_c^\infty(\Omega)$; we have $T_k(u_n - \phi) \in W^{1,p}_0(\Omega) $, and we replace $\phi$ with $T_k(u_n - \phi)$ in \eqref{lin:wsstep1}. We obtain
\begin{equation}\label{lin:eq:uepsvarphi}
\int_\Omega \Lambda  \nabla u_n \cdot \nabla T_k( u_n- \phi) \dx  + \epsilon_{n} \int_\Omega | \nabla  \psi  (u_n) |^{p-2} \nabla  u_n \cdot \nabla  T_k( u_n - \phi) \dx  = \int_{\Omega} f\,   T_k( u_n - \phi) \dx.
\end{equation}
Using the fact that the sequence $(u_n)_{n \ge 0}$ converges almost everywhere to $u$ and the fact that $T_k \in L^\infty(\R)$ we obtain
$$
\lim_{ n \to \infty}  \int_\Omega f(x) T_k( u_n - \phi) \dx =  \int_\Omega f(x) T_k( u - \phi) \dx.
$$
Applying Lemma \ref{lin:lem:liminfzv} below, we obtain
$$
\liminf_{n \to \infty} \int_\Omega \Lambda \nabla u_n \cdot \nabla T_k( u_n- \phi) \dx \ge \int_{\Omega} \Lambda\nabla u\cdot\nabla T_k(u- \phi) \dx.
$$
For the second term of the left member, we have
\begin{multline*}
 \epsilon_{n} \int_\Omega | \nabla  \psi  (u_n) |^{p-2} \nabla  u_n \cdot \nabla  T_k( u_n - \phi) \dx  = \epsilon_{n} \int_{\{|u_n - \phi| < k\}} | \nabla  \psi  (u_n) |^{p-2} \nabla  u_n \cdot \nabla  u_n  \dx\\
 - \epsilon_{n} \int_{\{|u_n - \phi| < k\}} | \nabla  \psi  (u_n) |^{p-2} \nabla  u_n \cdot \nabla  \phi\dx. 
\end{multline*}
Observe that the first term of the right-hand side of the above equation is nonnegative.
Applying the same estimates and H\"older's inequality as in the proof of Lemma \ref{lin:u:lem:weakcv}, we get that
$$
\lim_{n \to \infty} \epsilon_{n} \int_{\{|u_n - \phi| < k\}}| \nabla  \psi  (u_n) |^{p-2}  \nabla  u_n \cdot \nabla \phi \dx = 0,
$$
and therefore, that there holds
$$
\liminf_{n \to \infty}\epsilon_{n} \int_\Omega  | \nabla  \psi  (u_n) |^{p-2} \nabla  u_n \cdot \nabla  T_k( u_n - \phi)\dx\ge 0.
$$
This proves, letting $n\to+\infty$ in \eqref{lin:eq:uepsvarphi} with $n\in \mathbb{N}$, that  $u$ satisfies \eqref{eq:lin:entropyineq}, and is therefore the entropy weak solution in the sense of Definition \ref{def:lin:entsol}.
\end{proof}

\begin{lemma}\label{lin:lem:liminfzv} For the sequence provided by Lemma \ref{lin:u:lem:weakcv}, there holds, for all $k>0$,
\begin{equation}\label{lin:eq:liminfzv}
\int_{\Omega} \Lambda\nabla u\cdot\nabla T_k(u- \phi)  \dx \le \liminf_{n \to \infty} \int_\Omega \Lambda \nabla u_n \cdot \nabla T_k( u_n- \phi) \dx,\hbox{ for any }\phi \in C_c^\infty(\Omega).
\end{equation}
\end{lemma}
\begin{proof}
 We have
$$
\int_\Omega \Lambda  \nabla u_n \cdot \nabla T_k( u_n- \phi) \dx = 
\int_\Omega \Lambda  \nabla (u_n - \phi) \cdot \nabla T_k( u_n- \phi) \dx +\int_\Omega \Lambda  \nabla \phi \cdot \nabla T_k( u_n- \phi) \dx.
$$
Using the fact that $ \nabla T_k( u_n - \phi) = \nabla (u_n - \phi) 1_{\{|u_n - \phi| < k\}} $, and that $T_k'( s) = (T_k'( s))^2$, we have
\[
\int_\Omega \Lambda  \nabla (u_n-\phi) \cdot \nabla T_k( u_n- \phi) \dx =  \int_\Omega \Lambda   \nabla T_k (u_n-\phi) \cdot \nabla T_k( u_n- \phi) \dx.
\]
Using the weak convergence of the sequence $( \nabla T_k (u_n-\phi))_{n \ge 0} $ in $L^2(\Omega)^N$ to $T_k (u-\phi)$, we obtain
$$
\int_\Omega \Lambda \nabla T_k (u-\phi) \cdot \nabla T_k( u- \phi) \dx \le 
\liminf_{n \to \infty} \int_\Omega \Lambda  \nabla T_k (u_n-\phi) \cdot \nabla T_k( u_n- \phi) \dx.
$$
We also obtain
$$
\lim_{n \to \infty}\int_\Omega \Lambda \nabla \phi \cdot \nabla T_k( u_n- \phi) \dx =\int_\Omega \Lambda \nabla \phi \cdot \nabla T_k( u- \phi) \dx.
$$
We remark that
$$
\int_\Omega \Lambda \nabla T_k (u-\phi) \cdot \nabla T_k( u- \phi) \dx + \int_\Omega \Lambda \nabla \phi \cdot \nabla T_k( u- \phi) \dx = \int_{\Omega} \nabla T_k  (u- \phi) \cdot \nabla u \dx,
$$
which gives \eqref{lin:eq:liminfzv}.
\end{proof}

\section{The quasilinear problem}\label{sec:nonlinear}

\subsection{Origine and formulations}

Two quasilinear problems, which are extensions of Problem \eqref{eq:lin:probcont}, are classically involving measure data. One is the Richards problem, whose unknown is the pressure $w$ of the water phase within a porous medium containing air and water. It reads, in a simplified version, assuming that $\Lambda$ is the absolute permeability field,
\[
\partial_t\beta(w) - \text{div} ( \Lambda   \nabla w) = g  ~\text{in}~\Omega,
\]
where $\beta~:~\mathbb{R}\to [-1,0]$ is a nonstrictly increasing function (the quantity $\beta(w)+1\in [0,1]$ is called the ``water contents''), satisfying that $\beta(w) = 0$ for all $w\ge 0$. This problem is therefore a parabolic problem which degenerates into an elliptic one in the region where $w\ge 0$. The
right-hand-side represents injection or production terms, accurately modelled using measures along lines in 3D, points in 2D \cite{fab2000mod}.

A second example is the Stefan problems, whose unknown is the internal energy $w$ of a static material which is changing of state. Then the temperature is expressed as a function $\zeta(w)$, which is a nonstrictly increasing function, which remains constant in the range where $0\le w \le L$, where $L$ is the latent heat of change of state. Assuming that heat is provided by electric conductors, once again, a simplified model is

\[
\partial_t w - \text{div} ( \nabla \zeta(w)) = g  ~\text{in}~\Omega,
\]
in which the  right-hand-side is again accurately modelled using measures along lines in 3D, points in 2D.
So both problems can be cast into the common following problem: find a function $w$ such that

\[
\partial_t \beta(w) - \text{div} ( \Lambda   \nabla \zeta(w)) = g ~\text{in}~\Omega,
\]
where $\beta$ and $\zeta$ are two nonstrictly increasing functions, $\Lambda$ is a diffusion field and $f$ is a measure. Assuming that an implicit Euler scheme is used in time, which is done in most cases, wich consists in replacing $\partial_t \beta(w)$ by $(\beta(w)-\beta(w_{\rm prev}))/\delta\!t$ the semi-discrete problem to be solved, with respect to the  at each time step is then under the form

\[
\beta(w) - \text{div} ( \Lambda   \nabla \delta\!t\zeta(w)) = g +  \delta\!t \beta(w_{\rm prev})~\text{in}~\Omega,
\]

Applying the change of variable $v = (\beta+\zeta)(w)$, the problem becomes

\[
\beta\circ (\beta+\zeta)^{-1}(v) - \text{div} ( \Lambda   \nabla \delta\!t\zeta\circ (\beta+\zeta)^{-1}(v)) = g +  \delta\!t \beta(w_{\rm prev})~\text{in}~\Omega.
\]
Then we notice that the functions $ \beta\circ (\beta+\zeta)^{-1}$ and 
$\zeta\circ (\beta+\zeta)^{-1}$ are 1-Lipschitz continuous, the sum of which is equal to the identity function (see \cite{droniou2016conv}). We again denote $\beta,\zeta $ instead of  $ \beta\circ (\beta+\zeta)^{-1}$  and $\zeta\circ (\beta+\zeta)^{-1}$, and we have $\beta = {\rm Id} - \zeta$. Letting $\delta\!t = 1$, denoting by $f =  g +  \delta\!t \beta(w_{\rm prev})$, $b = \beta(v)$ and $u = \zeta(v)$, the problem is now to solve 
\[
 b - \text{div} ( \Lambda   \nabla u) = f,
\]
which is Problem \eqref{eq:probcont}-\eqref{eq:probbound}-\eqref{eq:probcontbetazeta}.
We therefore make the following assumption on the functions $\beta$ and $\zeta$:
\begin{subequations}
\begin{align}
\bullet~  & \zeta : \mathbb{R} \to \mathbb{R} \mbox{ is continuous and non-decreasing and 1-Lipschitz with }\zeta(0) = 0\mbox{ and }\nonumber\\
&\mbox{ there exist $Z_0>0$ and $Z_1>0$ such that $|\zeta(s)| \ge Z_1|s| - Z_0$ for any $s \in \mathbb{R}$.}\label{hypzeta}\\
\bullet~ & \beta = {\rm Id} - \zeta  \mbox{ is therefore continuous, non-decreasing and 1-Lipschitz with }\beta(0) = 0.\label{hypbeta}
\end{align}
\label{eq:hypbetazeta}
\end{subequations}

It is shown in \cite{droniou2016conv} that one can also plug Problem \eqref{eq:probcont}-\eqref{eq:probbound}-\eqref{eq:probcontbetazeta} into the maximal monotone graphs framework.
We define the graph $\mathcal{G}$ and the multivalued operator  $\mathcal{T}:\mathbb{R}\to \mathcal P(\mathbb{R})$ by
\begin{equation}\label{eq:defgraph}
\mathcal{G} = \{(\zeta(s),\beta(s)),s\in\mathbb{R}\}\hbox{ and } \mathcal{T}(s) = \{ t\in\mathbb{R}, (s,t)\in \mathcal{G} \},\mbox{ for all }s\in \mathbb{R}.
\end{equation}
We have the following properties (see \cite{droniou2016conv}): 
\begin{subequations}
\begin{align}
\bullet~  & \mathcal{T}\hbox{  is a maximal monotone operator with domain }\mathbb{R}\hbox{ such that }0\in \mathcal{T}(0),\label{eq:graphun}\\
\bullet~  & \hbox{ there exist }T_1, T_2, T_3, T_4\geq 0\hbox{ such that, for all  }x\in \mathbb{R}\hbox{  and all }y\in \mathcal{T}(x),\nonumber\\&T_{3}|x| - T_{4}\leq |y|\leq T_1|x|+T_2.\label{eq:graphdeux}
\end{align}
\label{eq:hypgraph}
\end{subequations}
It is then shown in \cite{droniou2016conv} that the function $\zeta$ can be identified as the resolvent of $\mathcal{T}$ defined by $(\rm Id + \mathcal{T})^{-1}$ and that \eqref{eq:probcontbetazeta} is equivalent to
\begin{equation}\label{eq:probcontgraph}
b(x)\in \mathcal{T}(u(x))\hbox{ for a.e. }x\in\Omega.
\end{equation}
Note that the maximal monotone graph setting \eqref{eq:probcontgraph} is used in \cite{bp05} for the study of renormalised solutions to the transient version of the problem studied in this paper. In \cite{bp05}, the additional assumption that the reciprocal graph $\mathcal{T}^{-1}$ is a continuous function is used in the existence theorem for identifying the pointwise limit of solutions to regularised problems using compactness arguments (this corresponds in our setting to assume that $\beta$, or equivalently $\rho$, is strictly increasing). 

\subsection{The regularised problem in the quasilinear case}\label{sec:regular}

Instead of writing 
\[
 b - \text{div} ( \Lambda   \nabla u) = f,
\]
with
\[
 b = \beta(v) \mbox{ and }u = \zeta(v),
\]
we use the technique provided in \cite{alt_luckhaus}: in order to express $v$ as a function of $u$, we introduce a given $\epsilon > 0$, and we modify the problem into
\[
 b = \beta(v) \mbox{ and }u = (\epsilon{\rm Id}+ \zeta)(v) = \epsilon v + \zeta(v).
\]
Since the function $\epsilon{\rm Id}+ \zeta$ is continuous, strictly increasing with image $\mathbb{R}$, we can then deduce that
\[
 v =  (\epsilon{\rm Id}+ \zeta)^{-1}(u) \mbox{ and }b=  \beta( (\epsilon {\rm Id}+\zeta)^{-1}(u)).
\]
We therefore consider the  following problem: defining the function $\mu_\epsilon$  by
\begin{equation}\label{eq:defbetaeps}
\forall s\in \mathbb{R}, \ \mu_\epsilon(s) = \beta( (\epsilon {\rm Id}+\zeta)^{-1}(s)),
\end{equation}
find a function $u$ defined on $\Omega$ such that, there holds in a weak sense,
\[
 \mu_\epsilon(u) - \text{div} ( \Lambda   \nabla u) = f.
\]
We now consider the techniques introduced in the preceding sections and we consider the problem
\[
 \mu_\epsilon(u_{\epsilon})- \Div ( \Lambda \nabla u_{\epsilon} +  \epsilon  |\nabla \psi(u_{\epsilon})|^{p-2}\nabla u_{\epsilon}) = f,
 \]
that is  Problem \eqref{lin:probconteps}-\eqref{lin:probboundeps}, in which $\mu_\epsilon$, which is defined by \eqref{eq:defbetaeps}, is continuous and nonstrictly increasing with $\mu_\epsilon(0) = 0$ (this property is the only one used on $\mu_\epsilon$ in the whole Section \ref{sec:regpb}).
Observe that $\epsilon$ plays a double role: it is used at the same time for regularising the dependence between $v$ and $u$ and for regularising the equation by addition of a weighted $p-$Laplace term.

\medskip

We can then directly apply Lemma \ref{lem:lin:existolreg}, which provides the existence of a solution to Problem \eqref{lin:probconteps}-\eqref{lin:probboundeps} in the sense of \eqref{lin:pbuepsilonreg}. The estimates provided by Lemmas \ref{lin:lem:estimep}, \ref{lin:lem:estlrep} and \ref{lin:lem:uepstendtoentsol} hold as well. In addition, accounting from the surlinearity property \eqref{hypzeta} of function $\zeta$, we obtain that the following lemma holds.

\begin{lemma}\label{lem:estimep}
Let  $\epsilon \in (0,Z_1/2)$. Let $u_\epsilon$ be given such that \eqref{lin:pbuepsilonreg} holds with $\mu_\epsilon$ defined by \eqref{eq:defbetaeps}. Then the function defined by
\begin{equation}\label{eq:defveps}
 v_\epsilon = (\epsilon {\rm Id}+\zeta)^{-1}(u_\epsilon)
\end{equation}
satisfies that there exists $\ctel{cte:veps}$ such that, for any ${\hat q}\in (1,+\infty)$ if $N=2$ and  ${\hat q}\in (1, N/(N-2))$ if $N\ge 3$
\begin{equation}\label{eq:esimveps}
\Vert v_\epsilon\Vert_{{\hat q}} \le \cter{cte:veps}.
\end{equation}
\end{lemma}
\begin{proof}
Since 
\[
 u_\epsilon =\epsilon v_\epsilon + \zeta(v_\epsilon),
\]
we get, applying \eqref{hypzeta}, 
\[
 |u_\epsilon| \ge |\zeta(v_\epsilon)| -\epsilon |v_\epsilon|  \ge -\epsilon |v_\epsilon| + Z_1|v_\epsilon|-Z_0 \ge \frac {Z_1} 2|v_\epsilon|-Z_0.
\]
Applying \eqref{lin:eq:estimpsi}, that is a bound on $\Vert u_\epsilon\Vert_{{\hat q}} $, this concludes the proof.

\end{proof}

\subsection{Convergence to a weak or entropy weak solution}\label{sec:qlconvergence}

Let us first state the weak sense for a solution $(b,u)$ to \eqref{eq:probcontbetazeta}-\eqref{eq:probbound}.

\begin{definition}[Weak solution to the quasilinear elliptic problem]\label{def:qlweaksol}
We say that  pair of measurable functions $(b,u) $ is a weak solution to Problem \eqref{eq:probcont}-\eqref{eq:probbound}-\eqref{eq:probcontbetazeta} if there exists a function  $v$ measurable on $\Omega$ such that $b = \beta(v)$ and $u = \zeta(  v )$  a.e. in $\Omega$  and
\begin{equation}
u  \in S_N(\Omega)\hbox{ and }\int_\Omega (b\, w +\Lambda\nabla u\cdot\nabla w) \dx = \int_\Omega w \d f,~\text{for any}~w \in \mathcal{T}_N(\Omega).
\label{u:eq:pbcons}\end{equation}
\end{definition}
Let us observe that $u \in S_N(\Omega)$ implies that $u \in L^r(\Omega)$ for any $r\in(1,\frac N {N-1})$. Using Assumption \eqref{hypzeta}-\eqref{hypbeta}, we deduce that $b \in L^r(\Omega)$ for any $r\in(1,\frac N {N-1})$. Note that all $w\in \mathcal{T}_N(\Omega)$ is an element of $C(\overline{\Omega})$.

We can now state a result of existence of a weak solution of the continuous problem, proved by passing to the limit on the regularised problem.

\begin{lemma} \label{u:lem:weakcv} Let  $(\epsilon_n)_{n \in \mathbb{N}} $ be a sequence of positive numbers  which converges to zero, and let  $u_{n}\in W^{1,p}_0(\Omega)$ be such that \eqref{lin:pbuepsilonreg} holds with $\epsilon = \epsilon_n$ and $\mu_{\epsilon_n}$ given by \eqref{eq:defbetaeps}.

Then there exist a subsequence of  $(\epsilon_n,u_n)_{n\in \mathbb{N}} $, again denoted $(\epsilon_n,u_n)_{n\in \mathbb{N}} $, $u\in \mathcal{S}_N(\Omega)$ and $v\in L^{\hat q}(\Omega)$, such that the sequence $(u_n)_{ n \in \mathbb{N}}$ converges to $u \in S_N(\Omega)$ weakly in $W^{1,q}_0(\Omega)$ for any $q\in (1, N/(N-1))$, strongly  in $L^{\hat q}(\Omega)$  and almost everywhere in $\Omega$ and $ v_n = (\epsilon_n {\rm Id}+\zeta)^{-1}(u_n)$ weakly converges to $v$ in $L^{\hat q}(\Omega)$ for all $\hat q\in [1,+\infty)$ if $N=2$ and for all $\hat q\in [1, N/(N-2))$ if $N\ge 3$.

Moreover, we have $u = \zeta(v)$ and, letting $b = \beta(v)$, the pair $(b,u)$ is a weak solution in the sense of Definition \ref{def:qlweaksol}, and, for any $ 1 < q < \frac{N}{N-1}$, there exists $\ctel{ql:cste:lru}$ only depending on $q$, $N$, $\underline{\lambda}$ and $\Vert f\Vert_{M(\Omega)}$ such that
 \[
 \| \nabla u\|_{q} \le \cter{ql:cste:lru} .
\]
\end{lemma}
\begin{proof}
Applying Lemmas \ref{lin:lem:estimep} and  \ref{lin:lem:estlrep}, we construct a subsequence of $(\epsilon_n,u_n)_{n\in \mathbb{N}} $ of the initial sequence, that we again denote identically, and we select  $u\in \mathcal{S}_N(\Omega)$ such that the chosen sequence $(u_n)_{ n \in \mathbb{N}}$ converges to $u \in S_N(\Omega)$ weakly in $W^{1,q}_0(\Omega)$ for any $q\in (1, N/(N-1))$, strongly  in $L^{\hat q}(\Omega)$ for all $\hat q\in [1,+\infty)$ if $N=2$ and for all $\hat q\in [1, N/(N-2))$ if $N\ge 3$,  and almost everywhere in $\Omega$.
 
\medskip 

Using Lemma \ref{lem:estimep}, since for $n$ large enough, we have $\epsilon_n\in (0,Z_1/2)$, we can extract from this sequence another subsequence, again denoted $(\epsilon_{n},u_n)_{n\in\mathbb{N}}$, such that $v_{n} = (\epsilon_n {\rm Id}+\zeta)^{-1}(u_n)$ (therefore  $ u_{n} = \zeta_{n}(v_{n})$ with $\zeta_n =\epsilon_n {\rm Id}+\zeta$)  converges to some $v\in {\mathcal S}_N(\Omega)$ for the weak topology of $L_{\hat q}(\Omega)$ with ${\hat q}\in (1,+\infty)$ if $N=2$ and  ${\hat q}\in (1, N/(N-2))$ if $N\ge 3$.

We have now to check that $u = \zeta(v)$ and that $\mu_{\epsilon_{n}}(u_{n})$ converges to $\beta(v)$ for the weak topology of $W^{1,p}_0(\Omega)$.

\medskip 

Applying Lemma \ref{lem:minty}, which states a consequence of Minty's trick, we indeed prove that $u = \zeta(v)$, since $\zeta_{n} = \zeta + \epsilon_{n}{\rm Id}$ satisfies the hypotheses of the lemma.

Turning to $\mu_{\epsilon_{n}}(u_{n})$, the relation 
\[
\forall s\in \mathbb{R}, \ \mu_\epsilon(s) = \beta( (\epsilon {\rm Id}+\zeta)^{-1}(s)) =  (\epsilon {\rm Id}+\zeta)^{-1}(s) - \zeta(\epsilon {\rm Id}+\zeta)^{-1}(s)) = (1+\epsilon)(\epsilon {\rm Id}+\zeta)^{-1}(s) - s.
\]
issued from \eqref{eq:defbetaeps}, leads to
\[
 \mu_{\epsilon_{n}}(u_{n}) = \beta(v_{n})= (1+\epsilon_{n})v_{n}- u_{n},
\]
which proves the convergence of $\mu_{\epsilon_{n}}(u_{n})$ to $v - \zeta(v) = \beta(v) =b$ for the weak topology of $W^{1,p}_0(\Omega)$.
\end{proof}

\begin{lemma}[Minty in $L^p$ space] \label{lem:minty}
Let $(\zeta_n)_{n\in\mathbb{N}}\subset C(\R)$ be a sequence of (non strictly) increasing Lipschitz-continuous function with the same Lipschitz constant $L_\zeta$, with $\zeta_n(0) = 0$ for all $n\ge 0$, which simply converges to a function $\zeta$ (which is therefore Lipschitz-continuous function with Lipschitz constant $L_\zeta$). Let ${q_1},{q_2}\in(1,2)$ be given, and let $(v_n)_{n\in\mathbb{N}}$ be a sequence of elements of $L^{q_2}(\Omega)$ such that
\begin{itemize}
 \item the sequence $(v_n)_{n\in\mathbb{N}}$ weakly converges to $v$ in $L^{q_2}(\Omega)$,
 \item the sequence $(\zeta_n(v_n))_{n\in\mathbb{N}}$ converges (strongly) to $u$ in $L^{q_1}(\Omega)$.
\end{itemize}
Then $u = \zeta(v)$ a.e. in $\Omega$.
\end{lemma}
\begin{proof} We first extract a subsequence of $(v_n,\zeta_n(v_n))_{n\in\mathbb{N}}$ a subsequence, again denoted $(v_n,\zeta_n(v_n))_{n\in\mathbb{N}}$, such that $\zeta_n(v_n)$ converges to $u$ a.e. and such that $|\zeta_n(v_n)|^{q_1}$ is dominated in $L^1(\Omega)$.
Let $\theta$ be such that $1<\theta \frac {{q_2}}{{q_2}-1} \le \min({q_1},{q_2})$, for any $s\in\R$, and let us denote by $P_\theta(s) = s^\theta$ if $s\ge 0$ and $P_\theta(s) = -(-s)^\theta$ if $s<0$.
This choice of $\theta$ ensures that:
 \begin{enumerate}
  \item $ | P_\theta(\zeta_n(v_n))|^{{q_2}/({q_2}-1)} \le \max(1,|\zeta_n(v_n)|)^{q_1}$ a.e., which implies that by dominated convergence $(P_\theta(\zeta_n(v_n)))_{n\in\mathbb{N}}$ converges in $L^{{q_2}/({q_2}-1)}(\Omega)$ to $P_\theta(u)$,
  \item for any $w\in L^{q_2}(\Omega)$, using $|\zeta_n(w)|\le L_\zeta|w|$ a.e., we have  $| P_\theta(\zeta_n(w))|^{{q_2}/({q_2}-1)} \le \max(1, (L_\zeta|w|)^{q_2})$ a.e., and, using the simple convergence of $\zeta_n$ to $\zeta$, we then get by dominated convergence that $(P_\theta(\zeta_n(w)))_{n\in\mathbb{N}}$ converges in $L^{{q_2}/({q_2}-1)}(\Omega)$ to $P_\theta(\zeta(w))$.
 \end{enumerate}
Using the fact that, for any $n \ge 0$, the function $P_\theta\circ\zeta_n $ is (nonstrictly) increasing, we obtain for any $w \in L^{q_2}(\Omega)$,
$$
\int_\Omega ( P_\theta(\zeta_n( v_n)) - P_\theta(\zeta_n(w))) ( v_n  -w) \dx \ge 0.
$$
Notice that, in the above expression, $P_\theta(\zeta_n( v_n)) - P_\theta(\zeta_n(w))\in L^{{q_2}/({q_2}-1)}(\Omega)$. It is then possible to let $n\to\infty$ in the previous inequality, which leads, by strong/weak convergence, to
$$
 \int_\Omega (P_\theta(u)  - P_\theta(\zeta(w))) (v-w) \dx \ge 0.
$$
We let $w = v + t \varphi$ where $ \varphi \in C_c^\infty(\Omega) $ and $t\in(0,1)$, and we obtain
$$
t  \int_\Omega (P_\theta(u)  - P_\theta(\zeta( v+t\varphi))) \varphi \dx \ge 0.
$$
Dividing by $t$ and using that $ |P_\theta(\zeta( v+t\varphi))|$ is dominated in $L^{1}(\Omega)$ by $ P_\theta(L_\zeta(|v|+|\varphi|))$, we obtain by letting $t\to 0$ and using dominated convergence
$$
 \int_\Omega  (P_\theta(u)  - P_\theta(\zeta( v))) \varphi \dx \ge 0.
$$
Since the above inequality also holds changing $\varphi$ in $-\varphi$, it is therefore an equality, which leads, since $\varphi$ is arbitrary, to 
$$
P_\theta(u)  - P_\theta(\zeta( v)) = 0~\text{a.e in}~\Omega.
$$
The previous identity thus gives
$$
u =\zeta(v)~\text{a.e in}~\Omega.
$$
\end{proof}

Let us now give the sense for the entropy solution  $(b,u)$ to \eqref{eq:probcontbetazeta}-\eqref{eq:probbound} in the case $f\in L^1(\Omega)$. The proof of uniqueness of this solution is done in Section \ref{sec:uniqsol}.

\begin{definition}[Entropy solution to the quasilinear elliptic problem with particular right-hand sides]\label{def:qlentsol}

We assume that $f\in L^1(\Omega)$.
We define an entropy solution of Problem \eqref{eq:probcont}-\eqref{eq:probbound}-\eqref{eq:probcontbetazeta} as a pair of measurable functions $(b,u) $ if there exists a function  $v$ measurable on $\Omega$ such that $b = \beta(v)$ and $u = \zeta(  v )$  a.e. in $\Omega$  and
\begin{enumerate}
 \item $u \in S_N(\Omega)$ and, for all $k>0$, $T_k(u)\in H^1_0(\Omega)$, where  $T_k$ is the truncation function defined by $T_k(s) = \min(|s|,k){\rm sign}(s)$ for all $s\in\R$ (where ${\rm sign}(s)=1$ if $s\ge 0$ and $-1$ if $s < 0$),
 \item the following holds
 \begin{equation}\label{entropyineq}
\int_{\Omega}b \, T_k(u -\phi)  \dx + \int_{ \Omega} \Lambda \nabla u \cdot \nabla T_k( u-\phi)  \dx  \le 
\int_{\Omega} f\,  T_k(u -\phi) \dx,
\end{equation}
for any $ \phi \in C_c^\infty(\Omega)$ and for any $k >0$.
\end{enumerate}
\end{definition}

We now turn to the existence result, the proof of which is again using the limit of the regularised problem.

\begin{lemma}\label{lem:uepstendtoentsol}
 Let $(b,u)$ be given by Lemma \ref{u:lem:weakcv}. Then, for all $k>0$, there exists $\ctel{cste:tku}$ only depending on $\underline{\lambda}$, $k$ and $ \| f\|_{1}$ such that
\[
 \| \nabla T_k (u)\|_{2} \le \cter{cste:tku}.
\]
Moreover, the pair $(b,u)$ is the entropy weak solution in the sense of Definition \ref{def:qlentsol}.
\end{lemma}
\begin{proof}
We follow the proof of Lemma \ref{lin:lem:uepstendtoentsol}. The only difference is the convergence of the first term.
This convergence  is a consequence of the convergence of $\int_\Omega \mu_\epsilon(u_\epsilon) T_k(u_\epsilon -  \phi)$ to $\int_\Omega b T_k(u -  \phi)$, owing to Lemma \ref{lem:estimep}, and to the fact that $ \mu_\epsilon(u_\epsilon)$ weakly converges to $b$ in $L^2(\Omega)$.
\end{proof}

\section{Uniqueness of the entropy solution for the quasilinear problem}\label{sec:uniqsol}

The following lemma enables the use of a larger test function space in the entropy weak sense.
\begin{lemma}[Test functions in $H^1_0(\Omega)\cap L^\infty(\Omega)$]\label{lem:hunlinf}
We assume that  Assumptions \eqref{eq:hypomlambf} and \eqref{eq:hypbetazeta} hold.
Let us assume that $f \in L^1(\Omega)$. Let $u$ be an entropy solution in the sense of Definition \ref{def:qlentsol}. Then \eqref{entropyineq} holds 
for any $ \phi \in H^1_0(\Omega)\cap L^\infty(\Omega)$ and for any $k >0$.
\end{lemma}
\begin{proof}
The proof follows the technique of \cite[Lemma 3.3]{ben1995theory}. From a sequence  $(\phi_n)_{n \ge 0} \in C^\infty_c(\Omega)$ converging to $\phi$ in $H^1_0(\Omega)$, one constructs a sequence, again denoted by $(\phi_n)_{n \ge 0} \in C^\infty_c(\Omega)$ such that  $(\phi_n)_{n \ge 0} $ is uniformly bounded by $M$, converges almost everywhere in $\Omega$ to $\phi$ and $ | \nabla \phi_n |$ is dominated in $L^2(\Omega)$. Then $T_k( u  -\phi_n)$ converges in $L^s(\Omega)$ to  $T_k( u  -\phi)$ for all $s\in[1,+\infty)$, $\nabla T_k( u  -\phi_n)$ weakly converges in $L^2(\Omega)^N$ to $\nabla T_k( u  -\phi)$.
We then remark that \eqref{entropyineq} yields
\begin{equation*}
\int_{\Omega}\beta(  u )  T_k( u  -\phi_n)  \dx + \int_{\Omega} \Lambda \nabla T_h( u ) \cdot \nabla T_k(  u -\phi_n)  \dx  \le 
\int_{\Omega} (f\,  T_k( u  -\phi_n)+ F\cdot\nabla T_k( u  -\phi_n)) \dx,
\end{equation*}
with $h = k + M$ (recall that $\nabla T_k( u  -\phi_n) = 0$ on the set  $ u >h$). We then let $n\to\infty$ in the above inequality, which gives \eqref{entropyineq} for any $ \phi \in H^1_0(\Omega)\cap L^\infty(\Omega)$.
\end{proof}
\begin{lemma}[An entropy weak solution is a weak solution]\label{lem:entisweak}
We assume that  Assumptions \eqref{eq:hypomlambf} and \eqref{eq:hypbetazeta} hold.
Let us assume that $f \in L^1(\Omega)$. Let $(b,u)$ be an entropy solution in the sense of Definition \ref{def:qlentsol}. Then $(b,u)$ is a weak solution is the sense of Definition \ref{def:weaksol}.
\end{lemma}
\begin{proof}
The proof follows the technique of \cite[Corollary 4.3]{ben1995theory}.
Owing to Lemma \ref{lem:hunlinf}, for given $\psi \in C^\infty_c(\Omega)$, $k>\Vert \psi\Vert_{L^\infty(\Omega)}$ and $h>0$, we can let $\phi = T_h( u )-\psi\in H^1_0(\Omega)\cap L^\infty(\Omega)$ in  \eqref{entropyineq}. This gives $A_1(h) + A_2(h) \le A_3(h)$ with
\begin{multline*}
A_1(h) = \int_{\Omega} b   T_k( u  -T_h( u )+\psi)  \dx,\\
A_2(h) =  \int_{ \Omega} \Lambda \nabla  u  \cdot \nabla T_k(  u -T_h( u )+ \psi)  \dx,\\
A_3(h) = \int_{\Omega}  T_k( u  -T_h( u )+\psi) f(x) \dx.
\end{multline*}
We observe that, defining $$
\chi_h(x) = 1 \mbox{ if }| u  -T_h( u )+\psi|<k\mbox{ and }0\mbox{ otherwise},
$$
we have
$$
A_2(h) = \int_{\Omega } \chi_h\Lambda \nabla  u  \cdot \nabla (  u -T_h( u )+ \psi)  \dx =A_{21}(h)+A_{21}(h),
$$
with
\[
A_{21}(h) =\int_{\Omega } \chi_h\Lambda \nabla  u  \cdot \nabla (  u -T_h( u ))  \dx\mbox{ and }
A_{22}(h) = \int_{\Omega } \chi_h \Lambda \nabla  u  \cdot \nabla \psi  \dx.
\]
Using the fact that the function $s\mapsto s-T_h(s)$ is (nonstrictly) increasing, we get that $ \nabla  u  \cdot \nabla (  u -T_h( u ))\ge 0$ a.e., and therefore $A_{21}(h)\ge 0$. 
We therefore obtain
\begin{equation}\label{eq:entisweak}
 A_1(h) + A_{22}(h) \le A_3(h)+A_4(h)\hbox{ for all }h>0.
\end{equation}
We now study the limit of \eqref{eq:entisweak}, letting $h\to +\infty$.
Since $\chi_h(x)$
converges to $1$ for a.e. $x\in \Omega$ as $h\to+\infty$ (recall that $k>\Vert \psi\Vert_{L^\infty(\Omega)}$), by dominated convergence, we get
\[
 \lim_{h\to +\infty}A_{22}(h) = \int_{\Omega } \Lambda \nabla  u  \cdot \nabla \psi  \dx.
\]
Using the fact that the sequence $ (T_k(  u  - T_h( u ) + \psi))_{h \ge 0} $ is bounded in $H_0^1(\Omega)$ and that $ T_h( u ) $ converges to $ u $ almost everywhere in $\Omega$, we obtain the weak convergence in $H_0^1(\Omega)$ of the sequence  $ (T_k(  u  - T_h( u ) + \psi))_{h \ge 0} $ to  $\psi $. This leads to
\[
\lim_{h\to +\infty}A_1(h) = \int_{\Omega} b  \psi  \dx\mbox{ and }
\lim_{h\to +\infty}A_3(h) = \int_{\Omega} \psi f(x) \dx,
\]
which enables to conclude that
\begin{equation*}
\int_\Omega ( b  \psi +\Lambda\nabla  u \cdot\nabla \psi) \dx \le \int_\Omega f\psi\dx.
\end{equation*}
Replacing $\psi$ by $-\psi$, we get that the above inequality is in fact an equality, which provides \eqref{u:eq:pbcons} for $w = \psi$. We then get  \eqref{u:eq:pbcons} for any $w \in \mathcal{T}_N(\Omega)$ by the density of $C^\infty_c(\Omega)$ in any $W^{1,r}_0(\Omega)$ for $r\in[1,+\infty)$.
\end{proof}

\begin{lemma} \label{lem:estimuniq}
We assume that  Assumptions \eqref{eq:hypomlambf} and \eqref{eq:hypbetazeta} hold.
Let  $(b,u)$ be an entropy solution in the sense of Definition \ref{def:qlentsol}. Then, for all $k>0$, there holds
\begin{equation}\label{eq:estimuniq}
 \lim_{h\rightarrow +\infty}\int_{h-k<| u |\le h+k} |\nabla u|^2\dx  = 0.
\end{equation}
\end{lemma}
\begin{proof}
Letting, for given $k,h>0$, $\phi =T_h( u )$ in \eqref{entropyineq} (this is possible thanks to Lemma \ref{lem:hunlinf}), we get
\[
\int_\Omega  b\, T_k( u -T_h( u )) \dx + \int_\Omega \Lambda \nabla  u \cdot \nabla T_k( u -T_h( u )) \dx
\leq
\int_\Omega f T_k( u -T_h( u )) \dx.
\]
Using $\nabla  u  = \nabla T_k( u -T_h( u ))$ for a.e. $x$ such that $\nabla T_k( u -T_h( u ))(x)\neq 0$, we get, denoting by $E_h = \{x\in \Omega, h< |u(x)|\le h+k\}$, 
\[
 \underline{\lambda} \Vert \nabla u\Vert_{L^2(E_h)}^2
\leq
\int_\Omega (f-  b ) T_k( u -T_h( u ))\dx,
\]
which gives
\[
 \underline{\lambda} \Vert \nabla u\Vert_{L^2(E_h)}^2
\leq
\int_{E_h} k (|f| +|b| )\dx.
\]
By dominated convergence, since $\chi_{E_h}(x)$ tends to $0$ a.e. as $h\to\infty$, we get
\[
 \lim_{h\to +\infty}\int_{E_h} k (|f| +|b| )\dx = 0,
\]
and therefore we obtain
\begin{equation}\label{eq:estimhk}
\lim_{h\to +\infty}\underline{\lambda} \Vert \nabla u\Vert_{L^2(E_h)}^2 = 0.
\end{equation}
Note that \eqref{eq:estimhk} implies that
$$
\lim_{h\rightarrow +\infty}\int_{h<| u |\le h+k} |\nabla  u |^2\dx
=
\lim_{h\rightarrow +\infty}\int_{h-k<| u |\le h} |\nabla  u |^2\dx
= 0,
$$
hence providing \eqref{eq:estimuniq}.
\end{proof}

\begin{lemma}[\cite{dallagio1996}]\label{lem:ineqiseq} 
We assume that  Assumptions \eqref{eq:hypomlambf} and \eqref{eq:hypbetazeta} hold.
Let  $(b,u)$ be an entropy solution in the sense of Definition \ref{def:qlentsol}.
Then, for  for any $k >0$ and for any $ \phi \in C_c^\infty(\Omega)$, there holds
 \begin{equation}\label{d:entropyeq}
\int_{\Omega}\Big( b T_k( u  -\phi) + \Lambda(x) \nabla  u  \cdot \nabla T_k(  u -\phi) \Big) \dx = \int_{\Omega}  T_k( u  -\phi) f(x) \dx.
\end{equation}
\end{lemma}
\begin{proof}
Let $\phi = 2 T_h( u )-\widetilde{\phi}$, for given $h>0$ and $ \widetilde{\phi} \in C_c^\infty(\Omega)$. Let $M = k + \Vert \widetilde{\phi}\Vert_\infty$. For $h>M$, we have:
\begin{itemize}
 \item $T_k( u  -2 T_h( u )+\widetilde{\phi}) =  u  + 2h +\widetilde{\phi}$ for $| u  + 2h +\widetilde{\phi}|\le k$,
 \item $T_k( u  -2 T_h( u )+\widetilde{\phi}) = - u  +\widetilde{\phi}$ for $| - u  +\widetilde{\phi}|\le k$,
 \item $T_k( u  -2 T_h( u )+\widetilde{\phi}) =  u  - 2h +\widetilde{\phi}$ for $| u  - 2h +\widetilde{\phi}|\le k$,
 \item otherwise  $T_k( u  -2 T_h( u )+\widetilde{\phi}) = \pm k$,
\end{itemize}
and we also have
\begin{equation}\label{eq:eqinterieur}
  T_k( u  -2 T_h( u )+\widetilde{\phi})  = T_k( - u  +\widetilde{\phi}) \hbox{ if }| u |\le 2h-M.
\end{equation}
This proves that
\[
 \int_{\Omega} \Lambda(x) \nabla  u  \cdot \nabla T_k(  u -\phi)  \dx =  \int_{\Omega} \Lambda(x) \nabla  u  \cdot \nabla T_k( - u  +\widetilde{\phi})  \dx + R_h,
\]
with 
\[
 |R_h| \le  \overline{\lambda}\int_{2h-M <  | u |<2h+M  }  |\nabla  u |(|\nabla  u | + |\nabla\widetilde{\phi}|) \dx.
\]
Applying Lemma \eqref{lem:estimuniq}, we get that
\[
 \lim_{h\to\infty} R_h=0.
\]
Besides, we get from \eqref{eq:eqinterieur} that
\[
 |\int_{\Omega}  T_k( u  -\phi) f(x) \dx - \int_{\Omega}  T_k( - u  +\widetilde{\phi}) f(x) \dx | \le 2k\int_{| u |\ge 2h-M} |f(x) -b(x)|\dx.
\]
By dominated convergence, we get that
\[
 \lim_{h\to\infty} \int_{| u |\ge 2h-M}  |f(x) -b(x)| \dx = 0.
\]
We consider $\phi = 2 T_h( u )-\widetilde{\phi}$ in \eqref{entropyineq}  (this is possible owing to Lemma \ref{lem:hunlinf}), and we obtain by letting $h\to\infty$,
\[
  \int_{\Omega}\Big( b T_k(- u  +\widetilde{\phi}) +  \Lambda(x) \nabla  u  \cdot \nabla T_k( - u  +\widetilde{\phi})\Big)  \dx\le  \int_{\Omega}  T_k( - u  +\widetilde{\phi}) f(x) \dx,
\]
which, in addition to  \eqref{entropyineq} with $\phi = \widetilde{\phi}$, provides \eqref{d:entropyeq}.
\end{proof}

\begin{theorem}\label{thm:uniq}
We assume that  Assumptions \eqref{eq:hypomlambf} and \eqref{eq:hypbetazeta} hold. Then
there exists an unique entropy solution to Problem \eqref{eq:probcont}-\eqref{eq:probbound}-\eqref{eq:probcontbetazeta} in the sense of Definition \ref{def:qlentsol}.
\end{theorem}
\begin{proof}
The proof follows that of \cite[Theorem 5.1]{ben1995theory}.
Let  $(b_1,u_1)$ and $(b_2,u_2)$ be two entropy solutions in the sense of Definition \ref{def:qlentsol}. We let, for $h,k>0$ be given, $\phi = T_h(u_i)$ in  \eqref{entropyineq} written for $(b_j,z_j)$  (one more time, this is possible thanks to Lemma \ref{lem:hunlinf}) and we add the resulting inequalities. We  get $A_1(h) + A_3(h) \le A_2(h)$ with
\begin{multline*}
A_1(h) = \int_{\Omega}( b_1  T_k( u_1 -T_h( u_2))+  b_2  T_k( u_2 -T_h( u_2)) )\dx,\\
A_2(h) = \int_{\Omega}  (T_k( u_1 -T_h( u_2))+T_k( u_2 -T_h( u_1)))  f  \dx,\\
A_3(h) =  \int_{ \Omega} (\Lambda \nabla  u_1 \cdot \nabla T_k(  u_1-T_h( u_2))+  \Lambda \nabla  u_2 \cdot \nabla T_k(  u_2-T_h( u_1)) )\dx.
\end{multline*}
Let us first study $A_1(h)$. We can write
\begin{align*}
&
\int_\Omega   b_1 T_k( u_1-T_h( u_2))\dx
\\ 
&
\qquad
=
\int_{| u_2|\leq h}   b_1 T_k( u_1- u_2)\dx
+
\int_{h<| u_2|}   b_1 T_k( u_1-T_h( u_2))\dx
\\
&
\qquad
=
\int_{| u_2|\leq h,| u_1|\leq h}   b_1 
T_k( u_1- u_2)\dx
\\
&
\qquad
\qquad
+
\int_{| u_2|\leq h,h<| u_1|}   b_1 
T_k( u_1- u_2)\dx
+
\int_{h<| u_2|}   b_1 T_k( u_1-T_h( u_2))\dx
\\
&
\qquad
\geq
\int_{| u_2|\leq h,| u_1|\leq h}   b_1 
T_k( u_1- u_2)\dx
-
\int_{| u_2|\leq h,h<| u_1|} |  b_1 |k \dx
-
\int_{h<| u_2|} |  b_1 |k \dx
\\
&
\qquad
\geq
\int_{| u_2|\leq h,| u_1|\leq h}   b_1 
T_k( u_1- u_2)\dx
-
\int_{h<| u_1|} |  b_1 |k \dx
-
\int_{h<| u_2|} |  b_1 |k \dx.
\end{align*}
Hence, applying the preceding computation to $(b_1,u_1)$ and $(b_2,u_2)$, and defining 
$$
\chi_h(x) = 1 \mbox{ if }(h<| u_1|,h<| u_2|)\mbox{ and }0\mbox{ otherwise},
$$
we get
\begin{align*}
&
\int_\Omega  ( b_1 T_k( u_1-T_h( u_2))
+
   b_2 T_k( u_2-T_h( u_1)))\dx
\\
&
\qquad
\geq
\int_{| u_2|\leq h,| u_1|\leq h} (  b_1 -  b_2 )
T_k( u_1- u_2)\dx
-
\int_{\Omega}\chi_h (|  b_1 |+|  b_2 |)k \dx.
\end{align*}
Using $ ( b_1- b_2)T_k( u_1- u_2)\ge 0$ which is a consequence of Assumption \eqref{eq:hypbetazeta}, we conclude that
\[
 A_1(h) \ge -\int_{\Omega}\chi_h (|  b_1 |+|  b_2 |)k \dx,
\]
which shows, by dominated convergence, that
\begin{equation}\label{eq:limaunh}
\liminf_{h\to+\infty}A_1(h)\ge 0. 
\end{equation}
Similar computations show that
\[
 A_2(h) \le 2 \int_{\Omega}\chi_h |f|\, k \dx,
\]
and therefore that
\begin{equation}\label{eq:limadeuxh}
\limsup_{h\to+\infty}A_2(h)\le 0.
\end{equation}

Following the analysis in \cite{ben1995theory}, we obtain that
\begin{equation}\label{eq:limatroish}
\liminf_{h\to +\infty}A_3(h) \geq\liminf_{h\to +\infty}\int_{|u_2|\leq h,|u_1|\leq h} \nabla(u_1-u_2) \nabla T_k(u_1-u_2)\dx \ge \int_{\Omega} \nabla(u_1-u_2) \nabla T_k(u_1-u_2)\dx.
\end{equation}
Gathering \eqref{eq:limaunh}, \eqref{eq:limadeuxh} and  \eqref{eq:limatroish}, we conclude that
\[
  \int_{\Omega} \nabla(u_1-u_2) \nabla T_k(u_1-u_2)\dx =0.
\]
Since the above relation holds for all $k>0$, we thus obtain that $\nabla(u_1-u_2) =0$ a.e. Using that $u_1$ and $u_2$ belong to $\mathcal{S}_N(\Omega)$, we conclude that $u_1=u_2$ a.e.

Applying Lemma \ref{lem:entisweak} for $(b_1,u_1)$ and $(b_2,u_2)$, we get that, for all $\phi\in C^\infty_c(\Omega)$,
\[
 \int_\Omega b_1\phi\dx =\int_\Omega b_2\phi\dx,
\]
which implies that $b_1 = b_2$ a.e. and concludes the proof of the uniqueness of the entropy solution.
\end{proof}

\bibliographystyle{abbrv}
\bibliography{linnonlin_bib}
\end{document}